\documentclass[12pt]{article}

\usepackage{amsmath}
\usepackage{amsthm}
\usepackage{graphicx}
\usepackage{psfrag,epsf}
\usepackage{amsfonts}
\usepackage{float}
\usepackage{flafter}
\usepackage{bbm}
\usepackage{enumitem}
\usepackage{natbib}
\usepackage[nottoc,numbib]{tocbibind}
\usepackage{booktabs,caption,subcaption}
\usepackage[flushleft]{threeparttable}
\usepackage{multirow}

\allowdisplaybreaks

\usepackage{color}
\usepackage{amssymb}
\usepackage{diagbox}

\newcommand{\bA}{\boldsymbol{A}}
\newcommand{\ba}{\boldsymbol{a}}
\newcommand{\bB}{\boldsymbol{B}}

\newcommand{\bC}{\boldsymbol{C}}

\newcommand{\bD}{\boldsymbol{D}}
\newcommand{\bE}{\boldsymbol{E}}

\newcommand{\bG}{\boldsymbol{G}}
\newcommand{\bg}{\boldsymbol{g}}
\newcommand{\bH}{\boldsymbol{H}}

\newcommand{\bj}{\boldsymbol{j}}

\newcommand{\bR}{\boldsymbol{R}}
\newcommand{\br}{\boldsymbol{r}}
\newcommand{\bS}{\boldsymbol{S}}
\newcommand{\bs}{\boldsymbol{s}}

\newcommand{\bV}{\boldsymbol{V}}
\newcommand{\bv}{\boldsymbol{v}}
\newcommand{\veco}{\mathbb{S}}
\newcommand{\bW}{\boldsymbol{W}}
\newcommand{\bw}{\boldsymbol{w}}
\newcommand{\bX}{\boldsymbol{X}}

\newcommand{\bY}{\boldsymbol{Y}}
\newcommand{\by}{\boldsymbol{y}}

\newcommand{\bZ}{\boldsymbol{Z}}
\newcommand{\balpha}{\boldsymbol{\alpha}}
\newcommand{\bbeta}{\boldsymbol{\beta}}

\newcommand{\boeta}{\boldsymbol{\eta}}
\newcommand{\bGamma}{\boldsymbol{\Gamma}}
\newcommand{\bgamma}{\boldsymbol{\gamma}}
\newcommand{\blambda}{\boldsymbol{\lambda}}
\newcommand{\bmu}{\boldsymbol{\mu}}

\newcommand{\bPsi}{\boldsymbol{\Psi}}
\newcommand{\bpsi}{\boldsymbol{\psi}}

\newcommand{\bSigma}{\boldsymbol{\Sigma}}
\newcommand{\bsigma}{\boldsymbol{\sigma}}
\newcommand{\bTau}{\boldsymbol{T}}
\newcommand{\btau}{\boldsymbol{\tau}}
\newcommand{\btheta}{\boldsymbol{\theta}}
\newcommand{\bTheta}{\boldsymbol{\Theta}}
\newcommand{\bzeta}{\boldsymbol{\zeta}}
\newcommand{\Zeta}{\Xi}

\newtheorem{theorem}{Theorem}
\newtheorem{corollary}{Corollary}
\newtheorem{lemma}{Lemma}

\theoremstyle{definition}

\newtheorem{remark}{Remark}
\newtheorem{thmlemma}{Lemma}[subsection]
\newenvironment{proofnoqed}{\par\noindent{\bf Proof\ }}{\\}

\makeatletter
\DeclareFontFamily{U}  {MnSymbolF}{}
\DeclareSymbolFont{symbolsMN}{U}{MnSymbolF}{m}{n}
\SetSymbolFont{symbolsMN}{bold}{U}{MnSymbolF}{b}{n}
\DeclareFontShape{U}{MnSymbolF}{m}{n}{
    <-6>  MnSymbolF5
   <6-7>  MnSymbolF6
   <7-8>  MnSymbolF7
   <8-9>  MnSymbolF8
   <9-10> MnSymbolF9
  <10-12> MnSymbolF10
  <12->   MnSymbolF12}{}
\DeclareFontShape{U}{MnSymbolF}{b}{n}{
    <-6>  MnSymbolF-Bold5
   <6-7>  MnSymbolF-Bold6
   <7-8>  MnSymbolF-Bold7
   <8-9>  MnSymbolF-Bold8
   <9-10> MnSymbolF-Bold9
  <10-12> MnSymbolF-Bold10
  <12->   MnSymbolF-Bold12}{}
\DeclareMathSymbol{\tbigtimes}{\mathop}{symbolsMN}{2}
\newcommand*{\bigtimes}{%
  \DOTSB
  \tbigtimes
  \slimits@ 
}
\makeatother

\newcounter{EE-conds}
\newcounter{thm-conds}

\begin{document}

\def\spacingset#1{\renewcommand{\baselinestretch}%
{#1}\small\normalsize} \spacingset{1}

\title{\bf Doubly Distributed Supervised Learning and Inference with High-Dimensional Correlated Outcomes}

\author{Emily C. Hector and Peter X.-K. Song \thanks{
We would like to acknowledge support for this project from the National Science Foundation (NSF DMS1513595) and the National Institutes of Health (NIH R01ES024732).} \hspace{.2cm}\\
Department of Biostatistics\\
University of Michigan}

\date{}

\maketitle

\begin{abstract}
This paper presents a unified framework for supervised learning and inference procedures using the divide-and-conquer approach for high-dimensional correlated outcomes. We propose a general class of estimators that can be implemented in a fully distributed and parallelized computational scheme. Modelling, computational and theoretical challenges related to high-dimensional correlated outcomes are overcome by dividing data at both outcome and subject levels, estimating the parameter of interest from blocks of data using a broad class of supervised learning procedures, and combining block estimators in a closed-form meta-estimator asymptotically equivalent to estimates obtained by \cite{Hansen}'s generalized method of moments (GMM) that does not require the entire data to be reloaded on a common server. We provide rigorous theoretical justifications for the use of distributed estimators with correlated outcomes by studying the asymptotic behaviour of the combined estimator with fixed and diverging number of data divisions. Simulations illustrate the finite sample performance of the proposed method, and we provide an R package for ease of implementation.
\end{abstract}

\noindent%
\vfill

\newpage

\spacingset{1.45} 

\section{INTRODUCTION}
\label{sec:intro}
Although the divide-and-conquer paradigm has been widely used in statistics and computer science, its application with correlated data has been little investigated in the literature. We provide a theoretical justification, with theoretical guarantees, for divide-and-conquer methods with correlated data through a general unified estimating function theory framework. In particular, in this paper we focus on the large sample properties of a class of distributed and integrated estimators for supervised learning and inference with high-dimensional correlated outcomes. We consider $N$ independent observations $\left\{ \by_i, \bX_i \right\}_{i=1}^N$ where both the sample size $N$ and the dimension $M$ of the response vector $\by_i$ may be so big that a direct analysis of the data using conventional methodology is computationally intensive, or even prohibitive. Such data may arise, for example, from imaging measurements of brain activity or from genomic data. Denote by $f(\bY_i; \bX_i, \btheta, \bGamma_i)$ the $M$-variate joint parametric distribution of $\bY_i$ conditioned on $\bX_i$, where $\btheta$ is the parameter of interest and $\bGamma_i$ contains parameters, such as for high-order dependencies, that may be difficult to model or handle computationally.\\
Statistical inference with big data can be extremely challenging due to the high volume and high variety of these data, as noted recently by \cite{Secchi}. In the statistics literature, methodological efforts to date have primarily focused on high-dimensional covariates (i.e. high-dimensional $\bX_i$) with univariate responses (corresponding to $M=1$); see \cite{Johnstone-Titterington} for an overview of the difficulties and methods in linear regression, and the citations therein for references to the extensive publications in this field. By contrast, little work has focused on high-dimensional correlated outcomes (corresponding to large $M$), which pose an entirely new and different set of methodological challenges stemming from a high-dimensional likelihood. The divide-and-combine paradigm holds promise in overcoming these challenges; see \cite{Mackey-Talwalkar-Jordan} and \cite{Zhang-Duchi-Wainwright} for early examples of the power of divide-and-combine algorithms. Some recent divide-and-combine methods for independent outcomes can be found in \cite{Singh-Xie-Strawderman}, \cite{Lin-Zeng}, \cite{Lin-Xi}, \cite{Chen-Xie}, and \cite{Liu-Liu-Xie}, among others.\\
More recently, \cite{Hector-Song} proposed a Distributed and Integrated Method of Moments (DIMM), a divide-and-combine strategy for supervised learning and inference in a regression setting with high-dimensional correlated outcomes $\bY$. DIMM splits the $M$ elements of $\bY$ into blocks of low-dimensional response subvectors, analyzes these blocks in a distributed and parallelized computational scheme using pairwise composite likelihood (CL), and combines block-specific results using a closed-form meta-estimator in a similar spirit to \cite{Hansen}'s seminal generalized method of moments (GMM). DIMM overcomes computational challenges associated with high-dimensional outcomes by running block analyses in parallel and combining block-specific results via a computationally and statistically efficient closed-form meta-estimator. DIMM is easily implemented using MapReduce in the Hadoop framework (\cite{Khezr-Navimipour}), where blocks of data are loaded only once and in parallel. DIMM presents a useful and natural extension of the classical GMM framework, which easily accounts for inter-block dependencies. DIMM also improves on the classical meta-estimation where results from blocks are routinely assumed to be independent. DIMM is still challenged, however, when estimating a homogeneous parameter in the presence of heterogeneous parameters. Additionally, it is also challenged computationally when the sample size $N$ is large; the strategy of dividing high-dimensional vectors of correlated outcomes into blocks is insufficient to address the excessive computational demand, since the sample size remains large in the block analyses. Thus, another division at the subject level is inevitable to mitigate the computational burden arising from matrix inversions and iterative calculations in the block analyses.\\
This paper proposes a new doubly divided procedure to learn and perform inference for a homogeneous parameter of interest in the presence of heterogeneous parameters with a general class of supervised learning procedures. The double division at the response and subject levels further speeds up computations in comparison to DIMM and results in a double division of the data, visualized in Table \ref{double-split}: a division of the response $\bY$, and a random division of subjects into independent subject groups, resulting in blocks of data with a smaller sample of low-dimensional response subvectors. We consider a general class of supervised learning procedures to analyze these blocks separately and in parallel. Then we establish a GMM-type combination procedure that yields a meta-estimator of the parameter of interest. This proposed estimator is more general than the DIMM estimator in \cite{Hector-Song}, and thus appealing in many practical settings where analyzing data with both large $M$ and $N$ is challenging. We achieve a doubly divided learning and inference procedure implemented in a distributed and parallelized computational scheme. The proposed class of supervised learning procedures is very general, including many important estimation methods as special cases, such as Fisher's maximum likelihood, \cite{Wedderburn}'s quasi-likelihood, \cite{Liang-Zeger}'s generalized estimating equations, \cite{Huber}'s M-estimation for robust inference, with possible extensions to semi-parametric and non-parametric models.\\
\begin{table}[h]
\centering
\captionsetup{justification=centering}
\scalebox{0.98}{
\begin{tabular}{c |c c c| c c| c c c} 
\diagbox[width=6em]{Block}{Group} & Subject 1 & $\ldots$ & Subject $n_1$ & $\ldots$ & $\ldots$ & Subject 1 & $\ldots$ & Subject $n_K$ \\ [0.5ex] 
 \hline
 1 & $y_{11,11}$ & $\ldots$ & $y_{n_11,11}$ & $\ldots$ & $\ldots$ & $y_{11,1K}$ & $\ldots$ & $y_{n_K1,1K}$\\
 \vdots & \vdots & \vdots & \vdots & \vdots & \vdots & \vdots &\vdots & \vdots \\
 $m_1$ & $y_{1m_1,11}$ & $\ldots$ & $y_{n_1m_1,11}$ & $\ldots$ & $\ldots$ & $y_{1m_1,1K}$ & $\ldots$ & $y_{n_Km_1,1K}$ \\
 \hline
 \vdots & \vdots & \vdots & \vdots & \vdots & \vdots & \vdots & \vdots & \vdots \\
 \vdots & \vdots & \vdots & \vdots & \vdots & \vdots & \vdots & \vdots &\vdots \\
 \hline
 1 & $y_{11,J1}$ & $\ldots$ & $y_{n_11,J1}$ & $\ldots$ & $\ldots$ & $y_{11,JK}$ & $\ldots$ & $y_{n_K1,JK}$\\ 
 \vdots & \vdots & \vdots & \vdots & \vdots & \vdots & \vdots & \vdots & \vdots \\
 $m_J$ & $y_{1m_J,J1}$ & $\ldots$ & $y_{n_1m_J,J1}$ & $\ldots$ & $\ldots$ & $y_{11,JK}$ & $\ldots$ & $y_{n_Km_J,JK}$ \\[1ex] 
\end{tabular}
}
\smallskip
\caption{Double division of outcome data on both the dimension of responses (into blocks) and sample size (into groups).}
\label{double-split}
\end{table}
\noindent The proposed Doubly Distributed and Integrated Method of Moments (DDIMM) not only provides a unified framework of various supervised learning procedures of parameters with heterogeneity under the divide-and-combine paradigm, but provides key theoretical guarantees for statistical inference, such as consistency and asymptotic normality, while offering significant computational gains when response dimension $M$ and sample size $N$ are large. These are useful and innovative contributions to the arsenal of tools for high-dimensional correlated data analysis, and to the collection of divide-and-combine algorithms, which have so far concentrated on independently sampled data. In this paper, we focus on the theoretical aspects of doubly distributed learning and inference, including a goodness-of-fit test based on a $\chi^2$ statistic. We also study consistency and asymptotic normality of the proposed estimator as the number of data divisions diverges. This includes theoretical justifications for distributed inference when the dimension of the response and the number of response divisions diverges, which allows the analysis of highly dense outcome data.\\
The rest of the paper is organized as follows. Section \ref{sec:methods} describes the DDIMM, with examples introduced in Section \ref{sec:examples}. Section \ref{sec:asymptotics} discusses large sample properties of the proposed DDIMM. Section \ref{sec:implementation} presents the main contribution of the paper, a closed-form meta-estimator and its implementation in a parallel and scalable computational scheme. Section \ref{sec:simulations} illustrates the DDIMM's finite sample performance with simulations. Section \ref{sec:discussion} concludes with a discussion. Additional proofs and simulation results are deferred to the Appendices and Supplemental Material. An R package is available in the Supplemental Material.

\section{FORMULATION}
\label{sec:methods}

We begin with some notation. Let $\left\| \cdot \right\|$ be the $\ell_2$-norm for a $D$-dimensional vector $\ba$ and a $D_1 \times D_2$-dimensional matrix $\bA$ defined by, respectively:
\begin{align*}
\begin{array}{cclcl}
\left\| \ba \right\| &=& \left( \sum \limits_{d=1}^D a_d^2 \right)^{1/2} &~~\mbox{ for }&\ba=\left[ a_d \right]_{d=1}^D \in \mathbb{R}^D,\\
\left\| \bA \right\| &=& \left( \sum \limits_{d_1=1}^{D_1} \sum \limits_{d_2=1}^{D_2} A_{d_1d_2}^2 \right)^{1/2} &~~\mbox{ for }&\bA=\left[ A_{d_1d_2} \right]_{d_1,d_2=1}^{D_1,D_2} \in \mathbb{R}^{D_1 \times D_2}.
\end{array}
\end{align*}
We define the stacking operator $\veco (\cdot)$ for matrices $\left\{ \bA_{jk} \right\}_{j=1,k=1}^{J,K}$, $\bA_{jk}\in \mathbb{R}^{D^{jk}_1\times D_2}$, as
\begin{align*}
\veco \left( \bA_{jk}, \bA_{j'k'} \right)&=\left( \begin{array}{cc}
\bA^T_{jk} & \bA^T_{j'k'}
\end{array} \right)^T 
\in \mathbb{R}^{ (D^{jk}_1+D^{j'k'}_1) \times D_2 },\\
\veco^J \left( \bA_{jk} \right)&=\left( \begin{array}{ccc}
\bA^T_{1k} & \ldots & \bA^T_{Jk} 
\end{array} \right)^T \in \mathbb{R}^{D^k_1 \times D_2},\\
\veco^{JK} \left( \bA_{jk} \right)&=\left( \begin{array}{ccccccc}
\bA^T_{11} & \ldots & \bA^T_{J1} & \ldots & \bA^T_{1K} & \ldots & \bA^T_{JK}
\end{array} \right)^T \in \mathbb{R}^{D_1 \times D_2},
\end{align*}
where $D^k_1=\sum_{j=1}^J D^{jk}_1$, $D_1=\sum_{k=1}^{K} D^k_1$. Consider the collection of samples $\left\{ \by_i, \bX_i \right\}_{i=1}^N$, where $\bX_i \in \mathbb{R}^{M \times q}$ is fixed, $\bY_i \in \mathbb{R}^M$, $q, M \in \mathbb{N}$. The number of covariates $q$ is considered fixed in this paper. Let $\btheta, \bzeta$ take values in parameter spaces $\Theta \subseteq \mathbb{R}^p$, $\Zeta \subseteq \mathbb{R}^d$, both compact subsets of $p$- and $d$-dimensional Euclidean space respectively. Let $p,d \in \mathbb{N}$, and consider $\btheta$ to be the parameter of interest, and $\bzeta$ to be a potentially large vector of parameters of secondary interest. Let $\btheta_0 \in \Theta, \bzeta_0 \in \Zeta$ be the true values of $\btheta$ and $\bzeta$ respectively. Consider a class $\mathcal{P}=\left\{ P_{\btheta, \bzeta} \right\}$ of parametric models with associated estimating functions $\bPsi$ of parameter $\btheta$ (e.g. $\bPsi$ can be the derivative of some objective function). Suppose we want to learn the parameter $\btheta$ by finding the root of $\bPsi(\btheta; \by, \bzeta)=\boldsymbol{0}$, which is computationally intensive or even prohibitive due to the large dimension $M$ of $\by$, the large sample size $N$, or the large dimension $d$ of $\bzeta$. We focus on a divide-and-combine approach utilizing modern distributed computing platforms to alleviate the computational and modelling challenges posed by analyzing the whole data.

\subsection{Double data split procedure}
\label{subsec:dds}

First, for each subject $i$, DDIMM divides the $M$-dimensional response $\by_i$ and its associated covariates into $J$ blocks, denoted by:
\begin{align*}
\by_i&=\left(\begin{array}{ccc} \by_{i,1}^T & \ldots & \by_{i,J}^T \end{array} \right)^T \mbox{ and }\bX_i=\left( \begin{array}{ccc} \bX_{i,1}^T & \ldots & \bX_{i,J}^T \end{array} \right)^T, ~i=1, \ldots, N.
\end{align*}
Division into blocks is not restricted to the order of data entry: responses may be grouped according to pre-specified block memberships, according to, say, substantive scientific knowledge, such as functional regions of the brain. In this paper, with no loss of generality, we use the order of data entry in the data division procedure. Further, DDIMM randomly splits the $N$ independent subjects to form $K$ disjoint subject groups $\left\{ \by_{i,jk}, \bX_{i,jk} \right\}_{i=1}^{n_k}$. Then each group has sample size $n_k$, $k=1, \ldots, K$, with $\sum_{k=1}^K n_k=N$. Refer to Table \ref{double-split} for notation detail. For ease of exposition, we henceforth use the term ``group'' to refer to the division along subjects, and ``block'' to refer to the division along responses. We also use the term ``block'' to refer to the division along both responses and subjects.\\ 
We call $\left\{ \by_{i,jk}, \bX_{i,jk} \right\}_{i=1}^{n_k}$ block $(j,k)$, $j=1, \ldots, J$ and $k=1, \ldots, K$. Within block $(j,k)$, let $m_j$ be the dimension of the sub-response, $\by_{i,jk}=(y_{i1,jk}, \ldots, \allowbreak y_{im_j, jk})^T \allowbreak \in \mathbb{R}^{m_j}$, and $\bX_{i,jk} \in \mathbb{R}^{m_j \times q}$ the associated covariate matrix, with $\sum_{j=1}^J m_j=M$. For each block $j \in \left\{1, \ldots, J \right\}$, we have $K$ independent subject groups $\left\{ \by_{i,jk} \right\}_{i=1, k=1}^{n_k,K}$. In contrast, each group $k \in \left\{1,\ldots, K \right\}$ has $n_k$ subjects and for each subject $i \in \left\{1, \ldots, n_k \right\}$, the $J$ response blocks $\left\{ \by_{i,jk} \right\}_{j=1}^{m_j}$ are dependent.\\ 
The primary task is to solve $\bPsi(\btheta; \by, \bzeta)=\boldsymbol{0}$ to learn parameter $\btheta$ ina  supervised way over the entire data. Given the above double data split scheme, this task becomes a divide-and-combine procedure: the first step is to solve the following system of block-specific estimating equations: for $j\in \left\{1, \ldots, J\right\}$, $k\in \left\{1, \ldots, K\right\}$,
\begin{align}
\bPsi_{jk}(\btheta; \by_{jk}, \bzeta_{jk})&=\boldsymbol{0},\label{block-EE-1}  \\
\bG_{jk}(\bzeta_{jk}; \by_{jk}, \btheta)&=\boldsymbol{0}, \label{block-EE-2}
\end{align}
where $\bG_{jk}$ is an estimating function used to learn parameters $\bzeta_{jk}$ (e.g. correlation parameters) that are allowed to be heterogeneous across blocks such that $\bzeta=\veco^{JK} \left( \bzeta_{jk} \right)$. The true values $(\btheta_0, \bzeta_{jk0})$ of $(\btheta, \bzeta_{jk})$ are the values such that $E_{\btheta_0, \bzeta_{jk0}} \veco ( 
\bPsi_{jk}(\btheta_0; \by_{jk}, \bzeta_{jk0}), \allowbreak
\bG_{jk}(\bzeta_{jk0}; \by_{jk}, \btheta_0)
)
= \boldsymbol{0}$. Parameters $\bzeta_{jk0}$ take values in parameter space $\Zeta_{jk} \subset \mathbb{R}^{d_{jk}}$ for some $d_{jk}>0$ such that $\bzeta_0=\veco ^{JK}\left( \bzeta_{jk0} \right)$, $\Zeta=\bigtimes_{j=1,k=1}^{J,K} \Zeta_{jk}$, $d=\sum_{k=1}^K \sum_{j=1}^J d_{jk}$. Let $\bzeta_{k0}=\veco^J \left( \bzeta_{jk0} \right)$ and $\bzeta_k=\veco^J \left(\bzeta_{jk} \right)$. This is a similar approach to GEE2, proposed by \cite{Zhao-Prentice}, with details also in \cite{Liang-Zeger-Qaqish}, where unbiased estimating equations for the nuisance parameters are added in order to guarantee consistency. In this way, we impose homogeneity of the parameter of interest $\btheta$ across blocks but allow heterogeneity of the parameters of secondary interest. We assume that the class of parametric models $\mathcal{P}$ yields block-specific estimating functions satisfying the following regularity assumptions:
\begin{enumerate}[label=(A.\arabic*), labelindent=0pt]
\item \label{consistent}
\begin{enumerate}[label=(\roman*), labelindent=0pt]
\item \label{unbiased}
$\bPsi_{jk}$ and $\bG_{jk}$ are unbiased; that is, for all $\btheta \in \bTheta$, $\bzeta_{jk} \in \Zeta_{jk}$, $E_{\btheta, \bzeta_{jk}} \veco ( 
\bPsi_{jk} (\btheta; \bY_{jk}, \bzeta_{jk}), \allowbreak
\bG_{jk} (\bzeta_{jk}; \bY_{jk}, \btheta)
)
=\boldsymbol{0}$.
\item $E_{\btheta_0, \bzeta_{jk0}} \veco \left( \bPsi_{jk}(\btheta; \bY_{jk}, \bzeta_{jk}), \bG_{jk}(\bzeta_{jk}; \by_{jk}, \btheta) \right)$ has a unique zero at $(\btheta_0, \bzeta_{jk0})$. \label{unique-0}
\item \label{additive}
$\bPsi_{jk}$ and $\bG_{jk}$ are additive: for some kernel inference functions $\bpsi_{jk}$ and $\bg_{jk}$, they take the form
\begin{align*}
\left( \begin{array}{c}
\bPsi_{jk}(\btheta; \by_{jk}, \bzeta_{jk}) \\
\bG_{jk}(\bzeta_{jk}; \by_{jk}, \btheta)
\end{array} \right)
&= \frac{1}{n_k} \sum \limits_{i=1}^{n_k} \left( \begin{array}{c}
\bpsi_{jk} (\btheta; \by_{i,jk}, \bzeta_{jk})\\
\bg_{jk}(\bzeta_{jk}; \by_{i,jk}, \btheta)
\end{array} \right).
\end{align*}
\end{enumerate}
\setcounter{EE-conds}{\value{enumi}}
\end{enumerate}
We define $\bPsi_{jk}$ and $\bG_{jk}$ as being ``weakly regular'' based on the above conditions \ref{consistent} \ref{unbiased}-\ref{additive} in which the defining properties of a regular inference function are applied to its mean; see \cite{Song} Chapter 3.5 for a definition of regular inference functions. Additional conditions on the class $\mathcal{P}$ will be described throughout the paper where appropriate. Within block $(j,k)$, denote by $\widehat{\btheta}_{jk}$ and $\widehat{\bzeta}_{jk}$ the joint solution to \eqref{block-EE-1} and \eqref{block-EE-2}, estimators of $\btheta$ and $\bzeta_{jk}$ respectively. For notation purposes, let $\widehat{\btheta}_{list}=\veco^{JK} ( \widehat{\btheta}_{jk})$, $\widehat{\bzeta}_k = \veco^J ( \widehat{\bzeta}_{jk} )$, and $\widehat{\bzeta}_{list} = \veco^{JK}(\widehat{\bzeta}_{jk})$. Due to the homogeneity of $\btheta$, the next step is integration of the block-specific estimators $\widehat{\btheta}_{jk}$. By contrast, $\widehat{\bzeta}_{jk}$ remain heterogeneous and potentially high-dimensional. In the rest of the paper, for convenience of notation, we suppress the dependence of $\bPsi_{jk}$, $\bG_{jk}$, $\bpsi_{jk}$ and $\bg_{jk}$ on $\by_{jk}$ and $\by_{i,jk}$:
\begin{align*}
\bPsi_{jk}(\btheta; \bzeta_{jk})&=\bPsi_{jk}(\btheta; \by_{jk}, \bzeta_{jk}),~~
\bG_{jk}(\bzeta_{jk}; \btheta)=\bG_{jk}(\bzeta_{jk}; \by_{jk}, \btheta),\\
\bpsi_{i,jk}(\btheta; \bzeta_{jk})&=\bpsi_{jk}(\btheta; \by_{i,jk}, \bzeta_{jk}),~~
\bg_{i,jk}(\bzeta_{jk}; \btheta)=\bg_{jk}(\bzeta_{jk}; \by_{i,jk}, \btheta).
\end{align*}

\subsection{Integration}

Integrating block estimates $\widehat{\btheta}_{jk}$ into an estimator of $\btheta$, denoted by $\widehat{\btheta}_c$, will yield a more efficient estimate of $\btheta$. In the integration step, our intuition is to treat each system of equations $\veco \left( \bPsi_{jk}(\btheta; \bzeta_{jk}), \bG_{jk}(\bzeta_{jk}; \btheta) \right)= \boldsymbol{0}$ as a ``moment condition'' on $\btheta$ contributed by block $(j,k)$, $j=1, \ldots, J$, $k=1, \ldots, K$. Technically, we want to derive an estimator $\widehat{\btheta}_c$ of $\btheta$ that satisfies all $JK$ moment conditions that effectively makes use of the $JK$ estimates of $\btheta$ obtained from equations \eqref{block-EE-1} and \eqref{block-EE-2}. To address the issue that $\btheta$ is over-identified by the $JK$ moment conditions, we invoke \cite{Hansen}'s seminal generalized method of moments (GMM) to combine the moment conditions that arise from each block. Another significant advantage of GMM is that it allows us to incorporate between-block dependencies, which cannot be easily done in classical meta-estimation. To this end, define the subject group indicator $\delta_i(k)=\mathbbm{1}(\mbox{subject $i$ is in blocks $(j,k)$ for some }k \in \left\{1, \ldots, K \right\} \mbox{ and for all }j=1, \ldots, \allowbreak J)$ for $i=1, \ldots, N$, $k=1, \ldots, K$. For subject $i$, let 
\begin{align*}
\bpsi_i (\btheta; \bzeta) = \veco^{JK} \left( \delta_i(k) \bpsi_{i,jk} (\btheta; \bzeta_{jk}) \right),~~
\bg_i (\bzeta; \btheta) = \veco^{JK} \left( \delta_i(k) \bg_{i,jk} (\bzeta_{jk}; \btheta) \right),
\end{align*}
where clearly only one $\veco^J \left(\delta_i(k) \bpsi^T_{i,jk}(\btheta; \bzeta_{jk}) \right)$ is non-zero. Let $\ba^{\otimes 2}$ denote the outer product of a vector $\ba$ with itself, namely $\ba^{\otimes 2}=\ba \ba^T$. Then we can define $\bPsi_N(\btheta; \bzeta)=(1/N) \sum_{i=1}^N \bpsi_i(\btheta; \bzeta) $. It is easy to show that
\begin{align*}
\bPsi_N(\btheta; \bzeta)&=\frac{1}{N} \veco^{JK} \left( \sum \limits_{i=1}^{n_k} \bpsi_{i,jk}(\btheta; \bzeta_{jk}) \right) =\frac{1}{N} \veco^{JK} \left( n_k \bPsi_{jk}(\btheta; \bzeta_{jk}) \right).
\end{align*}
Similarly, define $\bG_N(\bzeta; \btheta)=(1/N) \sum_{i=1}^N \bg_i(\bzeta; \btheta)=(1/N) \veco^{JK} \left( n_k \bG_{jk}(\bzeta_{jk}; \btheta) \right)$. Since $\bPsi_{jk}$ and $\bG_{jk}$ satisfy assumptions \ref{consistent} for each $j$ and $k$, $\bPsi_N$ and $\bG_N$ are additive, unbiased, and $E_{\btheta_0, \bzeta_0} \veco \left(\bPsi_N (\btheta; \bzeta), \bG_N(\bzeta; \btheta) \right)$ has a unique zero at $(\btheta_0,\bzeta_0)$. For convenience, we denote 
\begin{align}
\bTau_N (\btheta, \bzeta)&= \left( \begin{array}{c} \bPsi_N(\btheta; \bzeta) \\ \bG_N(\bzeta; \btheta) \end{array} \right),~~
\btau_i(\btheta, \bzeta)=\left( \begin{array}{c} \bpsi_i(\btheta; \bzeta) \\ \bg_i(\bzeta; \btheta) \end{array} \right).
\label{TauN}
\end{align}
We assume that the class $\mathcal{P}$ yields $\bpsi$, $\bg$ satisfying the following conditions:
\begin{enumerate}[label=(A.\arabic*), labelindent=0pt]
\setcounter{enumi}{\value{EE-conds}}
\item \begin{enumerate}[label=(\roman*), labelindent=0pt]
\item \label{psiN-conds-i} Both $\bpsi_{jk}$ and $\bg_{jk}$ are Lipschitz continuous in $\btheta$ and $\bzeta$, namely for $j \in \left\{1, \ldots, J \right\}$, $k \in \left\{1, \ldots, K \right\}$, and some constants $c_{jk}, b_{jk}>0$, for all $\left( \btheta_1, \bzeta_{jk1}\right), \left(\btheta_2, \bzeta_{jk2} \right)$ in a neighbourhood of $(\btheta_0, \bzeta_{jk0})$,
\begin{align*}
\left\| \bpsi_{i,jk}(\btheta_1; \bzeta_{jk1}) - \bpsi_{i,jk}(\btheta_2; \bzeta_{jk2}) \right\| &\leq c_{jk} \left\| (\btheta_1, \bzeta_{jk1}) - (\btheta_2, \bzeta_{jk2}) \right\|,\\
\left\| \bg_{i,jk}(\bzeta_{jk1}; \btheta_1) - \bg_{i,jk}(\bzeta_{jk2}; \btheta_2) \right\| &\leq b_{jk} \left\| (\btheta_1, \bzeta_{jk1}) - (\btheta_2, \bzeta_{jk2}) \right\|.
\end{align*} 
\item The sensitivity matrix $-\nabla_{\btheta, \bzeta} E_{\btheta, \bzeta} \btau_i(\btheta, \bzeta)$ is continuous in a compact neighbourhood $\mathbb{N}(\btheta_0, \bzeta_0)$ of $(\btheta_0$, $\bzeta_0)$, and positive definite; 
\item The variability matrix $E_{\btheta_0, \bzeta_0} \left( \btau_i(\btheta, \bzeta)^{\otimes 2} \right)$ is finite and positive-definite. 
\end{enumerate}
\label{psiN-conds}
\setcounter{EE-conds}{\value{enumi}}
\end{enumerate}
Note that $\bTau_N(\btheta, \bzeta)=\boldsymbol{0}$ has no unique solution because its dimension is bigger than the dimension of $\btheta$. To overcome this issue, we follow Hansen's GMM for over-identified parameters. Let $\bW$ be the weight matrix in the GMM equation \eqref{def:combined-estimator-nuisance}. Classical GMM theory states that any positive semi-definite matrix $\bW$ can be used to guarantee consistency and asymptotic normality of the resulting estimator, and that an optimal choice of $\bW$, corresponding to the inverse covariance of the estimating function $\bTau_N$ in \eqref{TauN}, leads to an efficient GMM estimator. In our setting, a possible formulation for a GMM estimator of $(\btheta, \bzeta)$ is
\begin{align}
(\widehat{\btheta}_c , \widehat{\bzeta}_c )&=
\arg \min \limits_{\btheta, \bzeta} Q_N(\btheta, \bzeta \lvert \bW), \mbox{ where} \label{def:combined-estimator-nuisance}\\
Q_N(\btheta, \bzeta \lvert \bW)&=\bTau^T_N(\btheta, \bzeta)
\bW
\bTau_N(\btheta, \bzeta). \nonumber
\end{align}
In \eqref{def:combined-estimator-nuisance}, the weight matrix $\bW$ is a positive semi-definite $(JKp+d) \times (JKp+d)$ matrix. The heterogeneity of $\bzeta$ allowed by the use of $\bG_N$ can lead to theoretical and computational challenges due to the high-dimensionality of the parameter, a problem from which GEE2 also suffers. See \cite{Chan-Kuk-Bell-McGilchrist} and \cite{Carey-Zeger-Diggle} for a discussion on the computational burden of inverting large matrices in GEE2. Note that block-specific estimators $\widehat{\bzeta}_{list}$ are consistent; the only possible improvement from re-learning $\bzeta$ in an iterative procedure between $\widehat{\btheta}_c$ and $\widehat{\bzeta}_c$ is a gain in efficiency. This is not necessary since $\bzeta$ are parameters of secondary interest and their efficiency is in general not of interest. We will derive a closed-form meta-estimator of $\btheta$ that avoids re-learning of $\bzeta$ in Section \ref{sec:implementation}.\\
Following the work of \cite{Hansen}, we define a particular instance of the estimator in \eqref{def:combined-estimator-nuisance} by specifying $\bW$ as the inverse sample covariance of $\bTau_N$. We will show in Section \ref{sec:asymptotics} that this choice of $\bW$ is optimal for the efficiency of the resulting estimator. Let $\widehat{\bV}_N$ be the sample covariance of $\bTau_N(\btheta_0, \bzeta_0)$:
\begin{align}
\widehat{\bV}_N&=\frac{1}{N} \sum \limits_{i=1}^N \left(\btau_i(\widehat{\btheta}_{list}, \widehat{\bzeta}_{list}) \right)^{\otimes 2} = \frac{1}{N} \sum \limits_{i=1}^N \left( \begin{array}{c} \bpsi_i(\widehat{\btheta}_{list}; \widehat{\bzeta}_{list}) \\ \bg_i(\widehat{\bzeta}_{list}; \widehat{\btheta}_{list}) \end{array} \right)^{\otimes 2}, \label{V_hat_N}
\end{align}
where $\bpsi_i(\widehat{\btheta}_{list}; \widehat{\bzeta}_{list})=\veco^{JK} \left( \delta_i(k) \bpsi_{i,jk}(\widehat{\btheta}_{jk}; \widehat{\bzeta}_{jk}) \right)$. Letting $\bW=\widehat{\bV}^{-1}_N$ yields the following optimal GMM estimator:
\begin{align}
(\widehat{\btheta}_{opt} , \widehat{\bzeta}_{opt} )&=
\arg \min \limits_{\btheta, \bzeta} 
\bTau^T_N(\btheta, \bzeta)
\widehat{\bV}^{-1}_N 
\bTau_N(\btheta, \bzeta). \label{def:combined-estimator-nuisance-opt}
\end{align}
We assume that $\bW$ and $\widehat{\bV}_N$ are nonsingular; see \cite{Han-Song} for optimal weighting matrix with QIF when the sample covariance is ill-defined. Before presenting large-sample properties of $(\widehat{\btheta}_c, \widehat{\bzeta}_c)$ and $(\widehat{\btheta}_{opt}, \widehat{\bzeta}_{opt})$ in Section \ref{sec:asymptotics}, we demonstrate in Section \ref{sec:examples} the flexibility of our framework through several important supervised learning methods.

\section{Examples}

\label{sec:examples}

We now present five examples to illustrate the flexibility of the unifying framework considered in this paper.

\subsection{Likelihood-based methods}
\label{subsec:likelihood}
Consider the multidimensional regression model $h(\bmu_{i,jk})=\bX_{i,jk} (\begin{array}{cc} \btheta^T & \bbeta^T_{jk} \end{array})^T$, where $\bmu_{i,jk}=E(\bY_{i,jk} \lvert \allowbreak \bX_{i,jk}, \btheta, \bbeta_{jk})$ is the mean vector of $\bY_{i,jk}$ given $\bX_{i,jk}$, $\bbeta_{jk}$, and the $p$-dimensional parameter $\btheta$ ($p \leq q$ the number of covariates, which may include an intercept), and $h$ is a known component-wise link function. Let $\bzeta_{jk}$ be parameters of the second-order moments of $\bY_{i,jk}$, such as dispersion parameters, and parameters in $\bbeta_{jk}$ (which may be empty). If the full likelihood of $\bY_{i,jk}$ is computationally tractable, $\bPsi_{jk}$ and $\bG_{jk}$ correspond to the score functions, and $\widehat{\btheta}_{jk}$ and $\widehat{\bzeta}_{jk}$ may be given by the maximum likelihood estimates (MLEs). DDIMM can be applied straightforwardly by following the procedure in Section \ref{sec:methods}.\\
If the full likelihood is computationally intractable or difficult to construct, one can instead use pseudo-likelihoods such as the pairwise composite likelihood. The pairwise composite likelihood, originally proposed by \cite{Lindsay} and detailed in \cite{Varin-Reid-Firth}, provides the following forms of the equations for \eqref{block-EE-1} and \eqref{block-EE-2}:
\begin{align*}
\bPsi_{jk}(\btheta; \bzeta_{jk})&=\frac{1}{n_k}\sum \limits_{i=1}^{n_k} \sum_{r=1}^{m_j-1} \sum_{t=r+1}^{m_j} \nabla_{\btheta} \log f_j (y_{ir,jk}; y_{it,jk}; \btheta, \bzeta_{jk}, \bX_{i,jk}),\\
\bG_{jk}(\bzeta_{jk}; \btheta)&=\frac{1}{n_k} \sum \limits_{i=1}^{n_k} \sum_{r=1}^{m_j-1} \sum_{t=r+1}^{m_j} \nabla_{\bzeta_{jk}} \log f_j (y_{ir,jk}; y_{it,jk}; \btheta, \bzeta_{jk}, \bX_{i,jk}),
\end{align*}
for some bivariate marginal $f_{j}$ which can be chosen according to the nature of the response data. As long as the bivariate marginals $f_j$ are correctly specified, the composite score functions $\bPsi_{jk}$ and $\bG_{jk}$ satisfy the regularity conditions in \ref{consistent}. Hence the DDIMM can be used to overcome the computational challenges related to the MLE and pairwise composite likelihood. We refer readers to Chapter 6 of \cite{Song} and Chapter 3 of \cite{Joe-2} for details on constructing multivariate distributions using Gaussian and vine copulas respectively, but note that direct computation of the MLE is computationally very challenging when $m_j \geq 4$. Examples of applications of Gaussian copulas can be found in \cite{Song-Li-Yuan}, \cite{Bodnar-Bodnar-Gupta}, \cite{Bai-Kang-Song}, and in the importance sampling algorithm proposed in \cite{Masarotto-Varin}, among others.

\subsection{Generalized estimating equations}
\label{subsec:gee}
More generally, \cite{Wedderburn}'s quasi-likelihood is a popular alternative method of supervised learning that does not require a fully specified multidimensional likelihood; it receives a full treatment in \cite{Heyde}. Consider \cite{Liang-Zeger}'s marginal mean model $h(\bmu_{i,jk})=\bX_{i,jk} (\begin{array}{cc} \btheta^T & \bbeta^T_{jk} \end{array})^T$ for the analysis of longitudinal data, where $\bmu_{i,jk}=E(\bY_{i,jk} \lvert \allowbreak \bX_{i,jk}, \btheta, \bbeta_{jk})$ is the marginal mean vector of serially correlated outcomes $\bY_{i,jk}$ given $\bX_{i,jk}$, $\bbeta_{jk}$, and the $p$-dimensional parameter $\btheta$ ($p \leq q$), and $h$ is a known component-wise link function. In this setting, $\bzeta_{jk}$ consists of parameters in $\bbeta_{jk}$ (which may be empty), parameters for the variances of $Y_{it,jk}$, $t=1, \ldots, m_j$, and a nuisance parameter $\balpha_{jk}$ which fully characterizes a working correlation matrix $\bR_{jk}(\balpha_{jk})$. In the case where $\bbeta_{jk}$ is empty, the generalized estimating equation (GEE) proposed by \cite{Liang-Zeger} yields the the kernel inference function $\bpsi_{jk}(\btheta; \bzeta_{jk}) =  \bD_{i,jk}^T \bSigma_{i,jk}^{-1} \br_{i,jk}$ in \ref{consistent} \ref{additive}, where $\bD_{i,jk}=\nabla_{\btheta} \bmu_{i,jk}$, $\br_{i,jk}=\by_{i,jk}-\bmu_{i,jk}$, and $\bSigma_{i,jk}=\bA_{i,jk} \bR_{jk}(\balpha_{jk}) \bA_{i,jk}$, where $\bA_{i,jk}=\mbox{diag} \left\{ (Var(Y_{it,jk}) )^{1/2} \right\}_{t=1}^{m_j}$. In GEE2, $\bG_{jk}$ in \eqref{block-EE-2} is specified as another unbiased inference function satisfying \ref{consistent} and \ref{psiN-conds}. DDIMM provides a procedure for the application of distributed methods to high-dimensional longitudinal/clustered data.

\subsection{M-estimation}
\label{subsec:m-est}
DDIMM can be applied to many learning methods proposed in robust statistics. In the robust statistics literature due to \cite{Huber} and, more generally, \cite{Huber-book}, an M-estimator is defined as the root of an implicit equation of the form $\bPsi_{jk}(\widehat{\btheta}_{jk})=\sum_{i=1}^{n_k} \bpsi_{jk}(\widehat{\btheta}_{jk})=\boldsymbol{0}$, where $\bpsi_{jk}(\btheta)=\nabla_{\btheta} \rho (\btheta)$, $\rho$ is a suitable function that primarily aims to provide estimators robust to influential data points, and $\widehat{\btheta}_{jk} \in \mathbb{R}^p$, and $\bzeta_{jk}$ is empty or known; additional details are available in \cite{Huber-book} for the case when $\bzeta_{jk}$ is unknown. In the context of longitudinal data, \cite{Wang-Lin-Zhu} robustify the generalized estimating equations of \cite{Liang-Zeger} by replacing the standardized residuals with Huber's $M$-residuals.

\subsection{Joint mean-variance modelling}
\label{subsec:joint}
Following \cite{Pan-Mackenzie}, one can jointly model the marginal means and covariances of the longitudinal responses with $h(\bmu_{i,jk})= \bX_{i,jk,1} \bbeta$, $\log (\bsigma^2_{i,jk})= \bX_{i,jk,2} \blambda$, and $\phi_{irt,jk}=\bX_{irt,jk,3} \bgamma$ for $1 \leq t<r \leq m_j$, where $h$ is a known component-wise link function, $\bbeta \in \mathbb{R}^{q_1}$, $\blambda \in \mathbb{R}^{q_2}$ and $\bgamma \in \mathbb{R}^{q_3}$ are unconstrained parameters,  $\bmu_{i,jk}=E(\bY_{i,jk} \lvert \bX_{i,jk,1}, \btheta)$ and $\bX_{i,jk,1} \in \mathbb{R}^{m_j \times q_1}$ a submatrix of $\bX_{i,jk}$, $\bsigma^2_{i,jk}=\veco \left( Var(Y_{ir,jk}) \right)_{r=1}^{m_j}$ and $\bX_{i,jk,2}\in \mathbb{R}^{m_j \times q_2}$ a submatrix of $\bX_{i,jk}$, and $\phi_{irt,jk}$ are specified in \cite{Zhang-Leng-Tang}. Estimating functions $\bPsi_{jk}$ and $\bG_{jk}$ in \eqref{block-EE-1} and \eqref{block-EE-2} are given in detail in \cite{Zhang-Leng-Tang}. There is some choice depending on the problem considered as to whether $\btheta=\bbeta$, $\btheta=( \blambda, \bgamma)$, or $\btheta=( \bbeta, \blambda, \bgamma )$. In the first case, learning of variance parameters only helps improve estimation efficiency. This type of framework is widely applied in biomedical studies where the mean parameters are of primary interest. In the second case, learning of covariance parameters is of interest and $\bbeta$ is treated as a nuisance parameter. This is the situation where prediction is of primary interest, such as in kriging in spatial data analysis. In the third case, $\bG_{jk}$ is null, and learning of variance parameters is of interest to the investigator. This case occurs for example in the study of volatility for risk management in financial data analysis.

\subsection{Marginal quantile regression for correlated data}
\label{subsec:quantile}
\sloppy
Consider the marginal quantile regression model $Q_{Y_{it,jk} \lvert \bX_{it,jk}}(\tau)=\bX_{it,jk} \btheta$, where $Q_{Y_{it,jk} \lvert \bX_{it,jk}}(\tau)=F^{-1}_{Y_{it,jk} \lvert \bX_{it,jk}}(\tau)=\inf \{ y_{it,jk} : F_{Y_{it,jk} \lvert \bX_{it,jk}}(y_{it,jk}) \geq \tau \}$ is the $\tau$th quantile of $Y_{it,jk} \lvert \bX_{it,jk}$, $\tau \in (0,1)$, where $f_{Y_{it,jk} \lvert \bX_{it,jk}}(y_{it,jk})$ is the conditional distribution function of $Y_{it,jk}$ given $\bX_{it,jk}$, $t=1, \ldots, m_j$. Many estimating functions $\bPsi_{jk}$ and $\bG_{jk}$ for the learning of $\btheta$ and association parameters $\bzeta_{jk}$ of $\bY_{i,jk}$ have been proposed; see \cite{Jung}, \cite{Fu-Wang}, \cite{Lu-Fan}, and \cite{Yang-Chen-Chang} for examples.\\

\noindent Each of these five examples requires additional work to fully develop a divide-and-conquer strategy via DDIMM, including specific computational details. Here we only present the general framework with a high-level discussion that sheds light on DDIMM's promising generality and flexibility, and its coverage of a wide range of supervised learning methods. The theoretical results presented in Sections \ref{sec:asymptotics} and \ref{sec:implementation} are developed under a general unified framework of estimating functions that includes the above five examples as special cases.

\section{ASYMPTOTIC PROPERTIES}
\label{sec:asymptotics}

In this section we assume that $K$ and $J$ are fixed; this assumption will be relaxed in Section \ref{sec:implementation}. Let $n_{\min}=\min_{k=1, \ldots, K} n_k$ and $n_{\max}=\max_{k=1, \ldots, K} n_k$. Suppose $\bW \stackrel{p}{\rightarrow} \bw$ as $n_{\min} \rightarrow \infty$. In this section we study the asymptotic properties of the GMM estimator $(\widehat{\btheta}_c, \widehat{\bzeta_c})$ proposed in \eqref{def:combined-estimator-nuisance} and its optimal version proposed in \eqref{def:combined-estimator-nuisance-opt}. We assume throughout that subjects are monotonically allocated to subject groups; that is, as $n_{\min} \rightarrow \infty$, a subject cannot be reallocated to another group once it has been assigned to a subject group. Define the variability matrix of $\btau_i(\btheta, \bzeta)$ in \eqref{TauN} as
\begin{align*}
\bv(\btheta, \bzeta)&=Var_{\btheta_0, \bzeta_0} \left\{ \btau_i (\btheta, \bzeta) \right\}=\left( \begin{array}{cc}
\bv_{\bpsi}(\btheta, \bzeta) & \bv_{\bpsi \bg}(\btheta, \bzeta)\\
\bv^T_{\bpsi \bg}(\btheta, \bzeta) & \bv_{\bg}(\btheta, \bzeta)
\end{array} \right)
\end{align*}
where $\bv_{\bpsi}(\btheta, \bzeta)=Var_{\btheta_0, \bzeta_0} \left\{ \bpsi_i(\btheta; \bzeta) \right\}$, $\bv_{\bg}(\btheta, \bzeta)=Var_{\btheta_0, \bzeta_0} \left\{ \bg_i(\bzeta; \btheta) \right\}$, and $
\bv_{\bpsi \bg}(\btheta, \bzeta)=E_{\btheta_0, \bzeta_0} \left\{ \bpsi_i(\btheta; \bzeta) \bg_i^T(\bzeta; \btheta) \right\}$.
Let the sensitivity matrix of $\btau_i(\btheta, \bzeta)$ be
\begin{align}
&~~~~\bs(\btheta, \bzeta)=-\nabla_{\btheta, \bzeta} E_{\btheta_0, \bzeta_0} \btau_i(\btheta, \bzeta)
=\left( \begin{array}{cc}
\bs^{\btheta}_{\bpsi}(\btheta, \bzeta) & 
\bs^{\bzeta}_{\bpsi}(\btheta, \bzeta)\\
\bs^{\btheta}_{\bg}(\btheta, \bzeta) & 
\bs^{\bzeta}_{\bg}(\btheta, \bzeta)
\end{array} \right), \mbox{ where}\\
\label{tau-sensitivity}
&\begin{array}{ll}
\bs^{\btheta}_{\bpsi}(\btheta, \bzeta)=\veco^{JK} \left( \frac{n_k}{N} \bs^{\btheta}_{\bpsi_{jk}}(\btheta, \bzeta_{jk}) \right), &
\bs^{\bzeta}_{\bpsi}(\btheta, \bzeta)=\mbox{diag} \left\{ \frac{n_k}{N}\bs^{\bzeta}_{\bpsi_{jk}}(\btheta, \bzeta_{jk}) \right\}_{j=1,k=1}^{J,K}, \nonumber \\
\bs^{\btheta}_{\bg}(\btheta, \bzeta)=\veco^{JK} \left( \frac{n_k}{N} \bs^{\btheta}_{g_{jk}}(\btheta, \bzeta_{jk}) \right),&
\bs^{\bzeta}_{\bg}(\btheta, \bzeta)=\mbox{diag} \left\{ \frac{n_k}{N} \bs^{\bzeta}_{\bg_{jk}}(\btheta, \bzeta_{jk}) \right\}_{j=1,k=1}^{J,K}
\end{array} \nonumber \\
&\bs_{jk}(\btheta, \bzeta_{jk})
=\left( \begin{array}{cc} 
\bs^{\btheta}_{\bpsi_{jk}}(\btheta, \bzeta_{jk}) & 
\bs^{\bzeta}_{\bpsi_{jk}}(\btheta, \bzeta_{jk})\\
\bs^{\btheta}_{\bg_{jk}}(\btheta, \bzeta_{jk}) & 
\bs^{\bzeta}_{\bg_{jk}}(\btheta, \bzeta_{jk})
\end{array} \right). \nonumber 
\end{align}
Following Theorem 3.4 of \cite{Song}, block-specific estimates $\widehat{\btheta}_{jk}$ and $\widehat{\bzeta}_{jk}$ are consistent given assumptions \ref{consistent}. Consistency and asymptotic normality of the GMM estimator $(\widehat{\btheta}_c, \widehat{\bzeta}_c)$ in \eqref{def:combined-estimator-nuisance} have been established by \cite{Hansen} and, more generally, by \cite{Newey-McFadden}. To establish consistency and asymptotic normality for the combined estimator $(\widehat{\btheta}_c, \widehat{\bzeta}_c)$, we consider the following additional regularity conditions:
\begin{enumerate}[label=(A.\arabic*)]
\setcounter{enumi}{\value{EE-conds}}
\item Following \cite{Newey-McFadden}, define
\begin{align*}
Q_0(\btheta, \bzeta \lvert \bW)=E_{\btheta, \bzeta} \left\{ \bTau^T_N(\btheta, \bzeta) \right\}
\bw
E_{\btheta, \bzeta} \left\{ \bTau_N(\btheta, \bzeta) \right\}.
\end{align*}
Assume $Q_0(\btheta, \bzeta \lvert \bW)$ is twice-continuously differentiable in a neighbourhood of $(\btheta_0, \bzeta_0)$. \label{cond-consist}
\item Let $(\widehat{\btheta}_c, \widehat{\bzeta}_c)=\arg \min \limits_{\btheta, \bzeta} Q_N(\btheta, \bzeta \lvert \bW)$. Following \cite{Newey-McFadden}, assume $Q_N(\widehat{\btheta}_c, \widehat{\bzeta}_c \lvert \bW) \leq \inf \limits_{\btheta \in \Theta, \bzeta \in \Zeta} Q_N(\btheta, \bzeta \lvert \bW) + \epsilon_N$ with $\epsilon_N=o_p(1)$. In addition, assume that $\btheta_0$, $\bzeta_0$ are interior points of $\Theta$ and $\Zeta$ respectively, and that for any $\delta_N \rightarrow 0$,
 \label{cond-norm}
\setcounter{EE-conds}{\value{enumi}}
\end{enumerate}
\begin{equation*}
\resizebox{\linewidth}{!} 
{
    $
\sup \limits_{\left\| (\btheta, \bzeta) - (\btheta_0, \bzeta_0) \right\| \leq \delta_N} \frac{N^{1/2}}{1+N^{1/2} \left\| (\btheta, \bzeta) - (\btheta_0, \bzeta_0) \right\| } \left\| \bTau_N(\btheta, \bzeta) - \bTau_N(\btheta_0, \bzeta_0) - E_{\btheta_0, \bzeta_0} \bTau_N(\btheta, \bzeta) \right\| \stackrel{p}{\rightarrow} 0.$
}
\end{equation*}
\\
Theorems \ref{thm:consist} and \ref{thm:norm} do not require the differentiability of $\bTau_{N}$ and $Q_N$. Instead, they require the differentiability of their population versions, and that $\bTau_N$ behaves ``nicely'' in a neighbourhood of $(\btheta_0, \bzeta_0)$, in the sense that higher order terms are asymptotically ignorable. The following two theorems state the consistency and asymptotic normality of $(\widehat{\btheta}_c, \widehat{\bzeta}_c)$ given in \eqref{def:combined-estimator-nuisance} under Newey and McFadden's mild moment conditions given in \ref{cond-consist} and \ref{cond-norm}. 
\begin{theorem}[Consistency of $(\widehat{\btheta}_c, \widehat{\bzeta}_c)$]
\label{thm:consist}
Suppose assumptions \ref{consistent}-\ref{cond-consist} hold with $(\widehat{\btheta}_c, \widehat{\bzeta}_c)$ defined in \eqref{def:combined-estimator-nuisance}. Then $(\widehat{\btheta}_c, \widehat{\bzeta}_c) \stackrel{p}{\rightarrow} (\btheta_0, \bzeta_0)$ as $n_{\min} \rightarrow \infty$. 
\end{theorem}
The proof of Theorem \ref{thm:consist} follows closely the steps given in \cite{Hansen} and \cite{Newey-McFadden}, and thus is omitted.
\begin{theorem}[Asymptotic normality of $(\widehat{\btheta}_c, \widehat{\bzeta}_c$]
\label{thm:norm}
Suppose assumptions \ref{consistent}-\ref{cond-norm} hold with $(\widehat{\btheta}_c, \widehat{\bzeta}_c)$ defined in \eqref{def:combined-estimator-nuisance}. Then as $n_{\min} \rightarrow \infty$,\\
\scalebox{0.95}{\parbox{\linewidth}{%
\begin{align*}
N^{1/2} \left( \begin{array}{c} \widehat{\btheta}_c - \btheta_0 \\ \widehat{\bzeta}_c - \bzeta_0 \end{array} \right) 
\stackrel{d}{\rightarrow} 
\mathcal{N} \left(0, 
\bj^{-1}(\btheta_0, \bzeta_0) \bs(\btheta_0, \bzeta_0) \tilde{\bv}(\btheta_0, \bzeta_0) \bs^T(\btheta_0, \bzeta_0) \bj^{-1}(\btheta_0, \bzeta_0)
\right),
\end{align*}
}}\\
where $\tilde{\bv}(\btheta, \bzeta)=\bw \bv(\btheta, \bzeta) \bw$, and where the Godambe information $\bj(\btheta, \bzeta)$ of $\bTau_N(\btheta, \bzeta)$ takes the form $\bj(\btheta, \bzeta)= \bs(\btheta, \bzeta) \bw \bs^T(\btheta, \bzeta)$.
\end{theorem}
\noindent The proof of Theorem \ref{thm:norm} follows easily from Theorem 7.2 in \cite{Newey-McFadden} and Theorem \ref{thm:consist} above. We note that requiring $K$ to be finite implies that $N$ and $n_{\min}$ are asymptotically of the same order. We will relax this assumption in Section \ref{sec:implementation}. Conditions \ref{cond-consist} and \ref{cond-norm} allow us to consider non-differentiable kernel inference functions in the block $(j,k)$ analysis, extending \cite{Hector-Song}'s DIMM beyond CL kernel inference functions. We can now consider quantile regression, M-estimation, and more general estimation functions than the score or CL score equations.\\
A test of the over-identifying restrictions follows from \cite{Hansen} and \cite{Hector-Song}. This test is useful for detecting invalid moment restrictions, which can inform our choice of data partition and model. Formally, we show in Theorem \ref{thm:test} that the objective function $NQ_N$ evaluated at $(\widehat{\btheta}_c, \widehat{\bzeta}_c)$ follows a $\chi^2$ distribution with $(JK-1)p$ degrees of freedom. Unfortunately, it may be difficult to tell if invalid moment restrictions stem from an inappropriate data split or incorrect model specification. Residual analysis for model diagnostics can remove doubt in the latter case.
\begin{theorem}[Test of over-identifying restrictions]
\label{thm:test}
Suppose assumptions \ref{consistent}-\ref{cond-norm} hold with $(\widehat{\btheta}_c, \widehat{\bzeta}_c)$ defined in \eqref{def:combined-estimator-nuisance}. Then as $n_{\min} \rightarrow \infty$, $NQ_N(\widehat{\btheta}_c, \widehat{\bzeta}_c \lvert \bW) \stackrel{d}{\rightarrow} \chi^2_{(JK-1)p}$.
\end{theorem}
\noindent The proof of Theorem \ref{thm:test} can be carried out with some minor changes from that of Theorem 3 in \cite{Hector-Song}. The GMM provides an objective function with which to do model selection even when the block analyses do not, such as with GEE and M-estimation. In the following, Theorem \ref{thm:weight-matrix} and Corollary \ref{cor-optimality} show our combined GMM estimator derived from \eqref{def:combined-estimator-nuisance-opt} is optimal in the sense defined by \cite{Hansen}: it has an asymptotic covariance matrix at least as small (in terms of the Loewner ordering) as any other estimator exploiting the same over-identifying restrictions. We refer to this property as ``Hansen optimality''.
\begin{theorem}
\label{thm:weight-matrix}
Suppose assumptions \ref{consistent}-\ref{psiN-conds} hold. Then as $n_{\min} \rightarrow \infty$, $\widehat{\bV}_{N} \stackrel{p}{\rightarrow} \bv(\btheta_0, \bzeta_0)$, i.e. $\bw=\bv^{-1}(\btheta_0, \bzeta_0)$.
\end{theorem}
\begin{proof}
The proof uses the consistency of the block estimators and the Central Limit Theorem, and is given in the Supplemental Material.
\end{proof}
\begin{corollary}[Hansen optimality]
\label{cor-optimality}
Suppose assumptions \ref{consistent}-\ref{cond-norm} hold with $(\widehat{\btheta}_c, \widehat{\bzeta}_c)$ defined in \eqref{def:combined-estimator-nuisance}. Let $\bj(\btheta, \bzeta)$ as given in Theorem \ref{thm:norm}. Then as $n_{\min} \rightarrow \infty$,
\begin{align*}
N^{1/2}  \left( \begin{array}{c} \widehat{\btheta}_{opt} -\btheta_0 \\ \widehat{\bzeta}_{opt} -\bzeta_0 \end{array} \right) \stackrel{d}{\rightarrow} \mathcal{N} \left(0, \bj^{-1}(\btheta_0, \bzeta_0)   \right).
\end{align*}
\end{corollary}
\noindent The theoretical results given in Theorems \ref{thm:consist}-\ref{thm:weight-matrix} provide a framework for constructing asymptotic confidence intervals and conducting hypothesis tests, so that we can perform inference for $\btheta$ when $M$ and/or $N$ are very large. Using an optimal weight matrix improves statistical power so DDIMM may detect some signals that other methods may miss. Since we consider a broad class of models $\mathcal{P}$, there are no general efficiency results about the block-specific estimator $\widehat{\btheta}_{jk}$. When a learning method based on $\bPsi_{jk}$ has known efficiency results and performs well enough, DDIMM generally inherits ``local'' efficiency to achieve overall efficiency.
\begin{remark}\leavevmode We discuss efficiency for selected examples in Section \ref{sec:examples}.
\begin{enumerate}[leftmargin=*, align=left, label=(\roman*), wide]
\item In Example \ref{subsec:likelihood}, when the score function exists and satisfies mild regularity conditions, its variance is given by Fisher's information, and is a lower bound on the variances of estimating functions for $\btheta$ and $\bzeta$. This, coupled with Hansen's optimality, means that using the score function for $\bpsi_{jk}$ and $\bg_{jk}$ yields an efficient estimator of $\btheta$ and $\bzeta$. In an unpublished dissertation, \cite{Jin} studied the efficiency of the pairwise composite likelihood under different correlation structures. \cite{Hector-Song} showed empirically that the efficiency of the pairwise composite likelihood propagates to the combined estimator.
\item In Example \ref{subsec:gee}, it is known that the GEE estimator $\widehat{\btheta}_{jk}$ in Example \ref{subsec:gee} is semi-parametrically efficient when the correlation structure of the response $\by_{i,jk}$ is correctly specified. This, coupled with Hansen's optimality, means that using GEE's for $\bpsi_{jk}$ with the correct correlation structure of the response $\by_{i,jk}$ yields an efficient estimator of $\btheta$.
\end{enumerate}
\end{remark}
\begin{remark}
The GMM estimator $(\widehat{\btheta}_{opt}, \widehat{\bzeta}_{opt})$ can be interpreted as maximizing an extension of the confidence distribution density, as discussed in \cite{Hector-Song}. The confidence distribution approach is used for independent data in \cite{Xie-Singh}. Briefly, we can define the confidence estimating function (CEF) as $U(\btheta, \bzeta)=\Phi ( N^{1/2} \widehat{\bV}_{N}^{-1/2} \bTau_N(\btheta, \bzeta) )$, where $\Phi(\cdot)$ is the $(JKp+d)$-variate standard normal distribution function. Clearly, $U(\btheta, \bzeta)$ is asymptotically standard uniform at $(\btheta_0, \bzeta_0)$ as long as $\widehat{\bV}_{N}$ is a consistent estimator of the covariance of $\bTau_N$. Then we can define the density of the CEF as $u(\btheta, \bzeta)=\phi ( N^{1/2} \widehat{\bV}_{N}^{-1/2} \bTau_N(\btheta, \bzeta) )$. Maximizing $u(\btheta, \bzeta)$ with respect to $(\btheta, \bzeta)$ yields the minimization defined in \eqref{def:combined-estimator-nuisance-opt}.
\end{remark}
\noindent By framing our estimator as a GMM estimator, the theoretical framework of DIMM established only for CL can be extended to include a data split at the subject level and a generalization of $\bPsi_{jk}$ and $\bG_{jk}$. Adding moment conditions allows the proposed method to enjoy the power and versatility of the GMM, and the necessary theoretical results to support its use. This divide-and-conquer strategy benefits from handling low dimensional blocks of data and estimating equations, yielding tremendous computational gains.

\section{DISTRIBUTED ESTIMATION AND INFERENCE}
\label{sec:implementation}

Despite the computational gains offered by the divide-and-combine procedure and the GMM estimator, iteratively finding the solution $(\widehat{\btheta}_{opt}, \widehat{\bzeta}_{opt})$ (or $(\widehat{\btheta}_c, \widehat{\bzeta}_c)$) to \eqref{def:combined-estimator-nuisance-opt} can be slow due to the high-dimensionality of parameter $\bzeta$ and the need to repeatedly evaluate $\bPsi_{jk}$ and $\bG_{jk}$. To overcome this computational bottleneck, we propose a meta-estimator derived from \eqref{def:combined-estimator-nuisance-opt} that delivers a closed-form estimator via a linear function of block estimates $(\widehat{\btheta}_{list}, \widehat{\bzeta}_{list})$. 
We define the DDIMM estimator for $(\btheta, \bzeta)$:
\begin{align}
\left( \begin{array}{c}
\widehat{\btheta}_{DDIMM}\\
\widehat{\bzeta}_{DDIMM}
\end{array} \right)&=\left( \sum \limits_{k=1}^K \sum \limits_{i=1}^J 
n^2_k \widehat{\bC}_{k,i}
\right)^{-1} 
\sum \limits_{k=1}^K \sum \limits_{i=1}^J n^2_k  \widehat{\bC}_{k,i} 
\left( \begin{array}{c} \widehat{\btheta}_{ik} \\ \widehat{\bzeta}_{list} \end{array} \right). 
\label{one-step-estimator-both}
\end{align}
where $\widehat{\bC}_{k,i}$ is a function of sample variability and sensitivity matrices and block-specific estimators $\widehat{\btheta}_{jk}$ and $\widehat{\bzeta}_{jk}$ defined in detail in Section \ref{subsec:implementation:def-C_ki}. If we do not plan to conduct inference for $\bzeta$, which is treated as a nuisance parameter, taking $\left[ \widehat{\bC}^{-1} \right]_{p:}$ to be rows $1$ to $p$ of matrix $( \sum_{k=1}^K \sum_{i=1}^J 
n^2_k \widehat{\bC}_{k,i}
)^{-1}$ leads to the closed-form estimator of $\btheta$:
\begin{align}
\widehat{\btheta}_{DDIMM}
&=\left[ \widehat{\bC}^{-1} \right]_{p:}
\sum \limits_{k=1}^K \sum \limits_{i=1}^J 
n^2_k  \widehat{\bC}_{k,i} 
\left( \begin{array}{cc} \widehat{\btheta}^T_{ik} & \widehat{\bzeta}^T_{list} \end{array} \right)^T. 
\label{one-step-estimator}
\end{align}
We briefly define sample sensitivity matrices that will appear in the main body of the paper. Let $\bS^{\btheta}_{\bpsi_{jk}}(\btheta, \bzeta_{jk})$ be a $n^{1/2}_k$-consistent sample estimator of $\bs^{\btheta}_{\bpsi_{jk}}(\btheta, \bzeta_{jk})$, and similarly define $\bS^{\bzeta}_{\bpsi_{jk}}(\btheta, \bzeta_{jk})$, $\bS^{\btheta}_{\bg_{jk}}(\btheta, \bzeta_{jk})$ and $\bS^{\bzeta}_{\bg_{jk}}(\btheta, \bzeta_{jk})$. Let
\begin{align*}
\bS_{jk}(\btheta, \bzeta_{jk})&=\left( \begin{array}{cc} 
\bS^{\btheta}_{\bpsi_{jk}}(\btheta, \bzeta_{jk}) & 
\bS^{\bzeta}_{\bpsi_{jk}}(\btheta, \bzeta_{jk})\\
\bS^{\btheta}_{\bg_{jk}}(\btheta, \bzeta_{jk}) & 
\bS^{\bzeta}_{\bg_{jk}}(\btheta, \bzeta_{jk})
\end{array} \right).
\end{align*}
Note that the uppercase $\bS$ denotes the sample sensitivity matrix, and the lower-case $\bs$ denotes the population sensitivity matrix. Let $\widehat{\bS}_{jk}=\bS_{jk}(\widehat{\btheta}_{jk}, \widehat{\bzeta}_{jk})$ and similarly define $\widehat{\bS}^{\btheta}_{\bpsi_{jk}}$, $\widehat{\bS}^{\bzeta}_{\bpsi_{jk}}$, $\widehat{\bS}^{\btheta}_{\bg_{jk}}$ and $\widehat{\bS}^{\bzeta}_{\bg_{jk}}$. Sensitivity formulas are summarized in Table \ref{sensitivity-summary} in Appendix \ref{sec:appendix:technical:sensitivity}. \\
The DDIMM estimator in \eqref{one-step-estimator} can be implemented in a fully parallelized and scalable computational scheme, where only one pass through each block of data is required. The block analyses are run on parallel CPUs, and return the values of summary statistics $\{ \widehat{\btheta}_{jk}, \widehat{\bzeta}_{jk}, \bpsi_{i,jk} (\widehat{\btheta}_{jk}; \widehat{\bzeta}_{jk}) , \allowbreak \bg_{i,jk}(\widehat{\bzeta}_{jk}; \widehat{\btheta}_{jk}), \allowbreak \widehat{\bS}_{jk} \}_{j,k=1}^{J,K}$ to the main computing node, which computes $\widehat{\btheta}_{DDIMM}$ in \eqref{one-step-estimator} in one step.

\subsection{Construction of $\widehat{\bC}_{k,i}$}
\label{subsec:implementation:def-C_ki}

We give details on the construction of $\widehat{\bC}_{k,i}$. Readers may wish to omit this section on a first reading, as these details are not necessary for an understanding of the main body of the paper. We consider the optimal case where the GMM weighting matrix takes the form:
\begin{align*}
\bW=\widehat{\bV}^{-1}_N=\left(\begin{array}{cc}
\widehat{\bV}_{N, \bpsi} & \widehat{\bV}_{N, \bpsi \bg} \\
\widehat{\bV}^T_{N, \bpsi \bg} & \widehat{\bV}_{N, \bg}
\end{array} \right)^{-1}=\left( \begin{array}{cc}
\widehat{\bV}^{\bpsi}_N & \widehat{\bV}^{\bpsi \bg}_N \\
\widehat{\bV}^{\bpsi \bg ~T}_N & \widehat{\bV}^{\bg}_N
\end{array} \right).
\end{align*}
For convenience, we introduce a subsetting operation, with technical details available in Appendix \ref{sec:appendix:technical:subset}: we let $\left[ \widehat{\bV}^{\bpsi}_N \right]_{ij:k}$ subset the rows for the parameters corresponding to block $(i,k)$ and the columns for the parameters corresponding to block $(j,k)$ of matrix $\widehat{\bV}^{\bpsi}_N$. Similarly define $\left[ \widehat{\bV}^{\bg}_N \right]_{ij:k}$, and $\left[ \widehat{\bV}^{\bpsi \bg}_N \right]_{ij:k}$. For $\boeta \in \left\{ \btheta, \bzeta \right\}$, let \\
\scalebox{0.96}{\parbox{\linewidth}{%
\begin{align*}
\widehat{\bA}^{\boeta}_{k,ij}&= 
\left( \widehat{\bS}^{\btheta~T}_{\bpsi_{jk}} \left[ \widehat{\bV}^{\bpsi}_N \right]_{ji:k} + \widehat{\bS}^{\btheta~T}_{\bg_{jk}} \left[ \widehat{\bV}^{\bpsi \bg~ T}_N \right]_{ji:k} \right) 
\widehat{\bS}^{\boeta}_{\bpsi_{ik}} + \left( \widehat{\bS}^{\btheta~T}_{\bpsi_{jk}} \left[ \widehat{\bV}^{\bpsi \bg}_N \right]_{ji:k} + \widehat{\bS}^{\btheta~T}_{\bg_{jk}} \left[ \widehat{\bV}^{\bg}_N \right]_{ji:k} \right) 
\widehat{\bS}^{\boeta}_{\bg_{ik}} ,\\
\widehat{\bB}^{\boeta}_{k,ij}&= \left( \widehat{\bS}^{\bzeta~T}_{\bpsi_{jk}} \left[ \widehat{\bV}^{\bpsi}_N \right]_{ji:k} + \widehat{\bS}^{\bzeta~T}_{\bg_{jk}} \left[ \widehat{\bV}^{\bpsi \bg~ T}_N \right]_{ji:k} \right)
\widehat{\bS}^{\boeta}_{\bpsi_{ik}} +\left( \widehat{\bS}^{\bzeta~T}_{\bpsi_{jk}} \left[ \widehat{\bV}^{\bpsi \bg}_N\right]_{ji:k} + \widehat{\bS}^{\bzeta~T}_{\bg_{jk}} \left[ \widehat{\bV}^{\bg}_N \right]_{ji:k} \right)
\widehat{\bS}^{\boeta}_{\bg_{ik}}.
\end{align*}
}}\\
\noindent Define $D^{ik}$ as the sum of the dimensions of $\bzeta_{11}, \ldots, \bzeta_{i-1k}$, and $D^k$ as the sum of the dimensions of $\bzeta_{11}, \ldots, \bzeta_{Jk-1}$, with technical details in Appendix \ref{sec:appendix:technical:Dik}. Let $d_k=\sum_{j=1}^J d_{jk}$. Then we can define the following,
\begin{align}
\widehat{\bC}_{k,i}&=\left( \begin{array}{cccc}
\sum \limits_{j=1}^J \widehat{\bA}^{\btheta}_{k,ij} & \boldsymbol{0}_{p\times D^{ik}} & \sum \limits_{j=1}^J \widehat{\bA}^{\bzeta}_{k,ij} & \boldsymbol{0}_{p\times (d-d_{ik}-D^{ik})} \\
\multicolumn{4}{@{}c@{\quad}}{\boldsymbol{0}_{D^{k} \times (p+d)}}\\
\widehat{\bB}^{\btheta}_{k,i1} & \boldsymbol{0}_{d_{1k}\times D^{ik}} & \widehat{\bB}^{\bzeta}_{k,i1} & \boldsymbol{0}_{d_{1k}\times (d-d_{ik}-D^{ik})} \\
\multicolumn{4}{@{}c@{\quad}}{ \vdots }\\
\widehat{\bB}^{\btheta}_{k,iJ} & \boldsymbol{0}_{d_{Jk}\times D^{ik}} & \widehat{\bB}^{\bzeta}_{k,iJ} & \boldsymbol{0}_{d_{Jk}\times (d-d_{ik}-D^{ik})} \\
\multicolumn{4}{@{}c@{\quad}}{ \boldsymbol{0}_{(d-d_k-D^{k}) \times (p+d)}}
\end{array} \right).
\label{def:C_ki}
\end{align}

\subsection{Asymptotic results for $K$ and $J$ fixed}
\label{subsec:implementation:asymptotics-fixed}

In this section we assume that $K$ and $J$ are fixed, which will be relaxed in Sections \ref{subsec:implementation:asymptotics-Kdiv} and \ref{subsec:implementation:asymptotics-KJdiv}. Recall that we assume subjects are monotonically allocated to subject groups: as $n_{\min} \rightarrow \infty$, a subject cannot be reallocated to another group once it has been assigned to a subject group. Consider the following condition:
\begin{enumerate}[label=(A.\arabic*)]
\setcounter{enumi}{\value{EE-conds}}
\item For each $j=1, \ldots, J$, $k=1, \ldots, K$, $\widehat{\btheta}_{jk}=\btheta_0+O_p(n_k^{-1/2})$ and $\widehat{\bzeta}_{jk}=\bzeta_{jk0}+O_p(n_k^{-1/2})$. For any $\delta_N \rightarrow 0$, \label{cond-equiv}
\setcounter{EE-conds}{\value{enumi}}
\end{enumerate}
\begin{equation*}
\resizebox{\linewidth}{!} 
{
    $
\sup \limits_{\left\| (\btheta, \bzeta) - (\btheta_0, \bzeta_0) \right\| \leq \delta_N} \frac{N^{1/2}}{1+N^{1/2} \left\| (\btheta, \bzeta) - (\btheta_0, \bzeta_0) \right\| } \left\| \bTau_N(\btheta, \bzeta) - \bTau_N(\btheta_0, \bzeta_0) - E_{\btheta_0, \bzeta_0} \bTau_N(\btheta, \bzeta) \right\| =O_p(N^{-1/2}).$
}
\end{equation*}\\
Consequently, some large-sample results can be established which are helpful in studying the asymptotic behaviour of $\widehat{\btheta}_{DDIMM}$.
\begin{lemma}
\label{lemma:add-results}
Suppose assumptions \ref{consistent}, \ref{psiN-conds} and \ref{cond-equiv} hold. Then we have consistent estimation of information matrices:
\begin{align*}
&\widehat{\bV}_{N}=\bv(\btheta_0, \bzeta_0) + O_p(N^{-1/2}),\\
&\widehat{\bS}_{jk}=\bs_{jk}(\btheta_0, \bzeta_{jk0})+O_p(n^{-1/2}_k)~~\mbox{ for each }j,k, \mbox{ and}\\
&\frac{1}{N^2} \sum \limits_{k=1}^K \sum \limits_{i=1}^J n^2_k \widehat{\bC}_{k,i} =\widehat{\bS}^T \widehat{\bV}^{-1}_N \widehat{\bS}=\bj(\btheta_0, \bzeta_0)+O_p(N^{-1/2}),\\
&\mbox{where}~~~\widehat{\bS}=\left( \begin{array}{cc}
\veco \left( \frac{n_k}{N} \widehat{\bS}^{\btheta}_{\bpsi_{jk}} \right)_{j=1,k=1}^{J,K} & \mbox{diag} \left\{ \frac{n_k}{N} \widehat{\bS}^{\bzeta}_{\bpsi_{jk}} \right\}_{j=1,k=1}^{J,K}\\
\veco \left( \frac{n_k}{N} \widehat{\bS}^{\btheta}_{g_{jk}} \right)_{j=1,k=1}^{J,K} & \mbox{diag} \left\{ \frac{n_k}{N} \widehat{\bS}^{\bzeta}_{g_{jk}} \right\}_{j=1,k=1}^{J,K}
\end{array} \right).
\end{align*}
\end{lemma}
\begin{proof}
A detailed proof is given in the Supplemental Material.
\end{proof}
We show in Theorem \ref{thm:dist-equiv} that the proposed closed-form estimator $(\widehat{\btheta}_{DDIMM}, \widehat{\bzeta}_{DDIMM})$ in \eqref{one-step-estimator-both} is consistent and asymptotically normally distributed.
\begin{theorem}
\label{thm:dist-equiv}
Suppose assumptions \ref{consistent}, \ref{psiN-conds} and \ref{cond-equiv} hold. Let $\bj(\btheta, \bzeta)$ as given in Theorem \ref{thm:norm}. As $n_{\min} \rightarrow \infty$,
\begin{align*}
N^{1/2} \left( \begin{array}{c} \widehat{\btheta}_{DDIMM}-\btheta_0 \\ \widehat{\bzeta}_{DDIMM}-\bzeta_0 \end{array} \right) \stackrel{d}{\rightarrow} \mathcal{N} \left(\boldsymbol{0}, \bj^{-1}(\btheta_0, \bzeta_0) \right).
\end{align*}
\end{theorem}
\begin{proof}[Proof of Theorem \ref{thm:dist-equiv}:] Here we present major steps, with all necessary details available in Appendix \ref{subsec:appendix:proofs:dist-equiv}. First, we show that $\widehat{\btheta}_{DDIMM}$ and $\widehat{\bzeta}_{DDIMM}$ are consistent. Define
\begin{align}
\lambda(\btheta, \bzeta)
&=\frac{1}{N^2} \sum \limits_{k=1}^K \sum \limits_{i=1}^J n^2_k \widehat{\bC}_{k,i} \left( \begin{array}{c} \btheta - \widehat{\btheta}_{ik} \\ \bzeta - \widehat{\bzeta}_{list} \end{array} \right).
\label{def:lambda}
\end{align}
By definition, $\lambda(\widehat{\btheta}_{DDIMM}, \widehat{\bzeta}_{DDIMM})=\boldsymbol{0}$. As shown in Lemma \ref{appendix:lemma:lambda} in Appendix \ref{subsec:appendix:proofs:dist-equiv}, $\lambda(\btheta_0, \bzeta_0) \allowbreak \stackrel{p}{\rightarrow} \boldsymbol{0}$ as $n_{\min} \rightarrow \infty$. Given that $\nabla_{\btheta, \bzeta} \lambda(\btheta, \bzeta)$ exists and is nonsingular, for some $(\btheta^*, \bzeta^*)$ between $(\widehat{\btheta}_{DDIMM}, \widehat{\bzeta}_{DDIMM})$ and $(\btheta_0, \bzeta_0)$, the first-order Taylor expansion leads to
\begin{align}
\lambda(\widehat{\btheta}_{DDIMM}, \widehat{\bzeta}_{DDIMM}) - \lambda(\btheta_0, \bzeta_0)
&=\nabla_{\btheta, \bzeta} \lambda(\btheta, \bzeta) \rvert_{\btheta^*,\bzeta^*} \left( \begin{array}{c} \widehat{\btheta}_{DDIMM}-\btheta_0 \\ \widehat{\bzeta}_{DDIMM} - \bzeta_0 \end{array} \right), 
\label{lambda-Taylor}
\end{align}
which converges in probability to $\boldsymbol{0}$ as $n_{\min} \rightarrow \infty$. This implies that $(\widehat{\btheta}_{DDIMM}, \allowbreak \widehat{\bzeta}_{DDIMM}) \stackrel{p}{\rightarrow} (\btheta_0, \bzeta_0)$ as $n_{\min} \rightarrow \infty$.\\
Now we derive the distribution of $(\widehat{\btheta}_{DDIMM}, \widehat{\bzeta}_{DDIMM})$. With a slight abuse of notation, let $\widehat{\btheta}_{list}-\btheta_0=\veco^{JK} \left( \widehat{\btheta}_{jk} -\btheta_0\right)$. We show in Lemma \ref{appendix:lemma:Taylor} in Appendix \ref{subsec:appendix:proofs:dist-equiv} that 
\begin{align}
\left( \begin{array}{c}
\bPsi_{jk}(\btheta_0; \bzeta_{jk0})\\
\bG_{jk}(\bzeta_{jk0}; \btheta_0)
\end{array} \right) &=\widehat{\bS}_{jk}\left( \begin{array}{c} \widehat{\btheta}_{jk}-\btheta_0 \\ \widehat{\bzeta}_{jk}-\bzeta_{jk0} \end{array} \right) + O_p(n^{-1}_k). \label{equiv-5-simp}
\end{align}
Recall the form of $\bTau_N$ in \eqref{TauN}. By the Central Limit Theorem, $N^{1/2} \bTau_N(\btheta_0, \bzeta_0) \stackrel{d}{\rightarrow} \mathcal{N} \left(0, \bv(\btheta_0, \bzeta_0) \right)$. Then with $\widehat{\bS}$ defined in Lemma \ref{lemma:add-results}, it follows from equation \eqref{equiv-5-simp} that
\begin{align*}
N^{1/2} \widehat{\bS} \left( \begin{array}{cc}
(\widehat{\btheta}_{list}-\btheta_0)^T &
(\widehat{\bzeta}_{list}-\bzeta_0)^T \end{array} \right)^T
\stackrel{d}{\rightarrow} \mathcal{N} \left(0, \bv(\btheta_0, \bzeta_0) \right).
\end{align*}
Moreover, by Lemma \ref{lemma:add-results} and Slutsky's theorem we have:
\begin{align*}
N^{1/2} \left( \begin{array}{cc}
(\widehat{\btheta}_{list}-\btheta_0)^T &
(\widehat{\bzeta}_{list}-\bzeta_0)^T \end{array} \right)^T
&\stackrel{d}{\rightarrow} \mathcal{N} \left(0,  \bj^{-1}(\btheta_0, \bzeta_0) \right).
\end{align*}
Using the fact that the sum of jointly (asymptotically) Normal variables is (asymptotically) normal, by Lemma \ref{lemma:add-results} and Slutsky's theorem again, we have\\
\scalebox{0.9}{\parbox{\linewidth}{%
\begin{align*}
&N^{1/2}\left( \begin{array}{c}
\widehat{\btheta}_{DDIMM}-\btheta_0\\
\widehat{\bzeta}_{DDIMM}- \bzeta_0 
\end{array} \right)=N^{1/2} \left( \sum \limits_{k=1}^K \sum \limits_{i=1}^J n^2_k \widehat{\bC}_{k,i} \right)^{-1} \sum \limits_{k=1}^K \sum \limits_{i=1}^J n^2_k \widehat{\bC}_{k,i} \left( \begin{array}{c} \widehat{\btheta}_{ik}-\btheta_0 \\ \widehat{\bzeta}_{list} - \bzeta_0 \end{array} \right) 
\end{align*}
}}\\
is asymptotically distributed $\mathcal{N}(\boldsymbol{0}, \bj^{-1}(\btheta_0, \bzeta_0))$.
\end{proof}
\noindent This key theorem allows us to use $\widehat{\btheta}_{DDIMM}$, which is more computationally attractive than $\widehat{\btheta}_{opt}$ defined in \eqref{def:combined-estimator-nuisance-opt}, without sacrificing any of the nice asymptotic properties for inference. Additionally, it follows easily from Theorem \ref{thm:dist-equiv} that, under suitable conditions, the closed-form estimator $(\widehat{\btheta}_{DDIMM}, \widehat{\bzeta}_{DDIMM})$ in \eqref{one-step-estimator-both} has the same asymptotic distribution as and is asymptotically equivalent to the GMM estimator $\widehat{\btheta}_{opt}$ in \eqref{def:combined-estimator-nuisance-opt}.
\begin{corollary}\label{cor:equiv}
Suppose assumptions \ref{consistent}-\ref{cond-equiv} hold with $(\widehat{\btheta}_{opt}, \widehat{\bzeta}_{opt})$ defined in \eqref{def:combined-estimator-nuisance-opt}. Then $(\widehat{\btheta}_{DDIMM}, \widehat{\bzeta}_{DDIMM})$ and $(\widehat{\btheta}_{opt}, \widehat{\bzeta}_{opt})$ are asymptotically equivalent: as $n_{\min} \rightarrow \infty$,
\begin{align*}
N^{1/2} \left\| \left( \begin{array}{c}
\widehat{\btheta}_{DDIMM}-\widehat{\btheta}_{opt}\\
\widehat{\bzeta}_{DDIMM}-\widehat{\bzeta}_{opt}
\end{array} \right)
 \right\| \stackrel{p}{\rightarrow}.
\end{align*}
\end{corollary}
\begin{proof} A detailed proof is given in the Supplemental Material.
\end{proof}
\noindent The computation of $\widehat{\btheta}_{DDIMM}$ in \eqref{one-step-estimator} relies solely on block-specific estimators $(\widehat{\btheta}_{list}, \widehat{\bzeta}_{list})$ and values of summary statistics from each block. To guarantee the appropriate asymptotic distribution of $\widehat{\btheta}_{DDIMM}$, we assume in condition \ref{cond-equiv} that these block-specific estimators are $N^{1/2}$ consistent estimators of the true values, which restricts the scope of possible block-specific inference methods. For inference methods not satisfying this $N^{1/2}$ consistency in condition \ref{cond-equiv}, it is still possible to use $\widehat{\btheta}_{opt}$ in \eqref{def:combined-estimator-nuisance-opt}. 

\subsection{Asymptotic results for diverging $K$ with $J$ fixed}
\label{subsec:implementation:asymptotics-Kdiv}

We show in Theorem \ref{thm:dist-inf} that the asymptotic distribution of $(\widehat{\btheta}_{DDIMM}, \widehat{\bzeta}_{DDIMM})$ remains unchanged as the number of subject groups $K$ grows with the sample size.
\begin{theorem}
\label{thm:dist-inf}
Suppose $N^{\delta-1/2}K$ is bounded as $n_{\min} \rightarrow \infty$ for a positive constant $\delta < \frac{1}{2}$,  and assumptions \ref{consistent}, \ref{psiN-conds} and \ref{cond-equiv} hold. Let $\bH \in \mathbb{R}^{h \times (p+d)}$ a matrix of rank $r \in \mathbb{N}$, $h \in \mathbb{N}$, $r \leq h$, with finite maximum singular value $\bar{\sigma}(\bH) < \infty$. Let $\bj(\btheta, \bzeta)$ as given in Theorem \ref{thm:norm}. Then, as $n_{\min} \rightarrow \infty$, we show that the limiting value $\bj_{\bH}(\btheta_0, \bzeta_0)$ of $\bH \bj^{-1}(\btheta_0, \bzeta_0) \bH^T$ is a positive semi-definite and symmetric variance matrix, and that
\begin{align*}
N^{1/2} \bH \left( \begin{array}{c}
\widehat{\btheta}_{DDIMM}-\btheta_0\\
\widehat{\bzeta}_{DDIMM}-\bzeta_0
\end{array} \right) 
&\stackrel{d}{\rightarrow} \mathcal{N} \left(\boldsymbol{0}, \bj_{\bH}(\btheta_0, \bzeta_0)  \right).
\end{align*}
\end{theorem}
\begin{proofnoqed}[Proof of Theorem \ref{thm:dist-inf}]
Here we present major steps, with all necessary details available in Appendix \ref{subsec:appendix:proofs:dist-inf}. First, we know that $\left\| \bH \right\| \leq r \bar{\sigma}(\bH)$. Let $\lambda(\btheta, \bzeta)$ defined by \eqref{def:lambda}, such that $\lambda(\widehat{\btheta}_{DDIMM}, \widehat{\bzeta}_{DDIMM})=\boldsymbol{0}$. We show in Lemma \ref{thm:dist-inf:lemma-1} in Appendix \ref{subsec:appendix:proofs:dist-inf} that $\left\| \lambda(\btheta_0, \bzeta_0) \right\| = O_p(N^{-1/2-\delta} n^{1/2}_{\max})$ and
$\left\| \left\{ \nabla_{\btheta, \bzeta} \lambda(\btheta, \bzeta) \right\}^{-1} \right\| = O_p\left( N^{1/2+\delta} n^{-1}_{\max} \right)$. From the first-order Taylor expansion in \eqref{lambda-Taylor}, we have
\begin{align*}
\left\| \bH \left( \begin{array}{c} \widehat{\btheta}_{DDIMM}-\btheta_0 \\ \widehat{\bzeta}_{DDIMM} - \bzeta_0 \end{array} \right) \right\|&\leq \left\| \bH \right\| \left\| \left( \nabla_{\btheta, \bzeta} \lambda(\btheta, \bzeta) \rvert_{\btheta^*,\bzeta^*} \right)^{-1} \right\| \left\| \lambda(\btheta_0, \bzeta_0) \right\|\\
&\leq r \bar{\sigma}(\bH) O_p(n^{-1/2}_{\max}).
\end{align*}
Then $\bH(\widehat{\btheta}^T_{DDIMM}, \widehat{\bzeta}^T_{DDIMM})^T - \bH (\btheta^T_0, \bzeta^T_0)^T \stackrel{p}{\rightarrow} \boldsymbol{0}$ as $n_{\min} \rightarrow \infty$.\\
To derive the distribution of $\bH(\widehat{\btheta}^T_{DDIMM}, \widehat{\bzeta}^T_{DDIMM})^T$, first consider an arbitrary $k \in \left\{1, \ldots, K \right\}$. For convenience, denote
\begin{align*}
\bTau_k(\btheta, \bzeta_k)&=\veco \left(
\veco^J \left( \bPsi_{jk}(\btheta; \bzeta_{jk})\right), 
\veco^J \left( \bG_{jk}(\bzeta_{jk}; \btheta) \right)
\right) ,\\
\btau_{i,k}(\btheta, \bzeta_k)&=\veco \left(
\veco^J \left( \bpsi_{i,jk}(\btheta; \bzeta_{jk})\right),
\veco^J \left( \bg_{i,jk}(\bzeta_{jk}; \btheta) \right)
\right).
\end{align*}
By the Central Limit Theorem, $n^{1/2}_k \bTau_k(\btheta_0, \bzeta_{k0})=n^{-1/2}_k \sum_{i=1}^{n_k} \btau_{i,k}(\btheta_0, \bzeta_{k0}) \stackrel{d}{\rightarrow} \mathcal{N} \left(\boldsymbol{0}, \bv_k (\btheta_0, \bzeta_{k0}) \right)$ as $n_k \rightarrow \infty$, where $\bv_k(\btheta, \bzeta_k)=Var_{\btheta_0, \bzeta_{k0}} \left\{ \btau_{i,k}(\btheta, \bzeta_k) \right\}$. Define
\begin{align*}
\bs_k(\btheta, \bzeta_k) &=\left( \begin{array}{cc}
\veco^J \left( \bs^{\btheta}_{\bpsi_{jk}}(\btheta, \bzeta_{jk}) \right) &
\mbox{diag} \left\{ \bs^{\bzeta}_{\bpsi_{jk}}(\btheta, \bzeta_{jk}) \right\}_{j=1}^J \\
\veco^J \left( \bs^{\btheta}_{g_{jk}}(\btheta, \bzeta_{jk}) \right) & \mbox{diag} \left\{ \bs^{\bzeta}_{\bg_{jk}}(\btheta, \bzeta_{jk}) \right\}_{j=1}^J
\end{array} \right), \mbox{ and}\\
\bj_k(\btheta, \bzeta_k)&=\bs^T_k(\btheta, \bzeta)  \bv_k^{-1}(\btheta, \bzeta_k) \bs_k(\btheta, \bzeta_k).
\end{align*}
By \eqref{equiv-5-simp} in the proof of Theorem \ref{thm:dist-equiv}, Lemma \ref{lemma:add-results}, and Slutsky's theorem,
\begin{align*}
n^{1/2}_k  \bj_k(\btheta_0, \bzeta_{k0})  \left( \begin{array}{c}
\veco \left( \widehat{\btheta}_{jk} -\btheta_0 \right)_{j=1}^J\\
\widehat{\bzeta}_{k} - \bzeta_{k0}
\end{array} \right) \stackrel{d}{\rightarrow} \mathcal{N} \left(\boldsymbol{0}, \bj^{-1}_k(\btheta_0, \bzeta_{k0}) \right).
\end{align*}
Note that the above vectors are independent for $k=1, \ldots, K$. We establish in Lemma \ref{thm:dist-inf:lemma-2} in Appendix \ref{subsec:appendix:proofs:dist-inf} that, for some affine transformation matrices $\bE_k$, $k=1, \ldots, K$, of $\boldsymbol{0}$'s and $\boldsymbol{1}$'s,
\begin{align*}
\frac{n^2_k}{N^2} \sum \limits_{i=1}^J \widehat{\bC}_{k,i}\left( \begin{array}{c} \widehat{\btheta}_{ik}-\btheta_0 \\ \widehat{\bzeta}_{list} - \bzeta_0 \end{array} \right)
&=\frac{n_k}{N} \bE_k \bZ_k + O_p\left( N^{-1} \right),\\
\mbox{and}~~~\frac{n^2_k}{N^2} \sum \limits_{i=1}^J \widehat{\bC}_{k,i}&=\frac{n_k}{N} \bE_k \bj_k(\btheta_0, \bzeta_{k0}) \bE^T_k + O_p\left( n^{1/2}_kN^{-1} \right),
\end{align*}
where $n^{1/2}_k  \bZ_k \stackrel{d}{\rightarrow} \mathcal{N} \left(\boldsymbol{0}, \bj^{-1}_k(\btheta_0, \bzeta_{k0}) \right)$. It is clear that $\bj(\btheta, \bzeta)=\sum_{k=1}^K (n_k/N) \allowbreak \bE_k \bj_k(\btheta, \bzeta_k) \bE^T_k$. Since $\bE_k$ has finitely many $1$'s, $\left\| \bE_k \right\|$ is bounded. Since $\left\| \bj_k(\btheta, \bzeta_k) \right\|$ is also bounded, $\left\| \bj(\btheta, \bzeta) \right\|= O(K n_{\max} N^{-1})=O(1)$. $\bj(\btheta_0, \bzeta_0)$ is positive semi-definite and symmetric, implying that $\bH \bj^{-1}(\btheta_0, \bzeta_0) \bH^T$ is also positive semi-definite and symmetric. Following the monotone convergence theorem, we can write $\bH \bj^{-1}(\btheta_0, \bzeta_0) \bH^T \rightarrow \bj_{\bH} (\btheta_0, \bzeta_0)$, where $\bj_{\bH} (\btheta_0, \bzeta_0)$ exists and is a proper variance matrix.\\
Using the fact that $\lambda(\widehat{\btheta}_{DDIMM}, \widehat{\bzeta}_{DDIMM})=\boldsymbol{0}$ and $K=O(N^{1/2-\delta})$, we show in Lemma \ref{thm:dist-inf:lemma-3} in Appendix \ref{subsec:appendix:proofs:dist-inf} that $N^{1/2} \bH ( \widehat{\btheta}_{DDIMM}-\btheta_0, \widehat{\bzeta}_{DDIMM}- \bzeta_0 )$ can be rewritten as
\begin{align*}
\bH \left\{ \sum \limits_{k=1}^K \frac{n_k}{N} \bE_k \bj_k(\btheta_0, \bzeta_{k0}) \bE^T_k + O_p\left(n^{1/2}_{\max} N^{-1/2-\delta}\right) \right\} ^{-1}& \\
\left[ \sum \limits_{k=1}^K \left\{ \left( \frac{n_k}{N} \right)^{1/2} \bE_k n^{1/2}_k \bZ_k \right\} + O_p \left(N^{-\delta} \right) \right]&.
\end{align*}
Since $O_p(n^{1/2}_{\max}N^{-1/2-\delta}) =o_p(1)$ and $O_p(N^{-\delta}) =o_p(1)$, it follows as in the proof of Theorem \ref{thm:dist-equiv} that as $n_{\min}\rightarrow \infty$,
\begin{align*}
N^{1/2} \bH \left( \begin{array}{c}
\widehat{\btheta}_{DDIMM}-\btheta_0\\
\widehat{\bzeta}_{DDIMM}- \bzeta_0 
\end{array} \right) \stackrel{d}{\rightarrow} \mathcal{N} \left(\boldsymbol{0}, \bj_{\bH}(\btheta_0, \bzeta_0) \right).&
\qed
\end{align*}
\end{proofnoqed}
\noindent In practice, Theorem \ref{thm:dist-inf} suggests that we can tune our choice of $K$ and $n_{\min}$ to attain the desired trade-off between inference and computational speed: smaller $K$ and larger $n_{\min}$ will slow computations but improve estimation and asymptotic normality, whereas larger $K$ and smaller $n_{\min}$ will speed computations but worsen estimation and asymptotic normality. 

\subsection{Asymptotic results for diverging $K$ and $J$}
\label{subsec:implementation:asymptotics-KJdiv}

In general, asymptotics for diverging $J$ become very complicated and even analytically intractable depending on how, and to what extent, the dependence structure evolves as the dimension $M$ of $\bY$ goes to infinity ($M \rightarrow \infty$). \cite{Cox-Reid} propose constructing a pseudolikelihood from marginal densities when the full joint distribution is difficult to construct, and discuss asymptotics for increasing response dimensionality. To make the problem of diverging $M$ tractable, we consider the following regularity conditions:
\begin{enumerate}[label=(A.\arabic*)]
\setcounter{enumi}{\value{EE-conds}}
\item Stationarity: for each $M^* \in \mathbb{N}$ and each $(M^*+1)$-dimensional measurable set $B$ a subset of the sample space of $\bY$, the distribution of $\bY_i$ satisfies $P\left\{ (Y_{i,r}, \ldots, Y_{i,r+M^*}) \in B \right\}=P\left\{ (Y_{i,0}, \ldots, Y_{i,M^*} ) \in B \right\}$ for every $r \in \mathbb{N}$.
\label{cond-diverging-J-3.1}
\item Let $\bC_{k,i}$ be the version of $\widehat{\bC}_{k,i}$ in \eqref{def:C_ki} evaluated at the true values $\btheta_0, \bzeta_{jk0}$. For $k =1, \ldots, K$, $i=1, \ldots, J$, $( \sum_{l=1}^K \sum_{j=1}^J n^2_l \bC_{l,j})^{-1} n_k^2 \bC_{k,i}=O_p(N^{-\delta_1} )$ for a constant $0 \leq \delta_1 \leq 1/2$. This can be thought of as a type of Lindeberg condition.
\label{cond-diverging-J-3.2}
\item Conditions required for asymptotically normal distribution and efficiency of the GMM estimator $(\widehat{\btheta}_{opt}, \widehat{\bzeta}_{opt})$; see Theorem 5.4 in \cite{Donald-Imbens-Newey} and the spanning condition in \cite{Newey}. See \cite{Newey} for related work on semiparametric efficiency of the GMM estimator as the number of moment conditions goes to infinity. 
\label{cond-diverging-J-3.3}
\setcounter{EE-conds}{\value{enumi}}
\end{enumerate}
\begin{remark}
Condition \ref{cond-diverging-J-3.1} is typical for consistency and asymptotic normality of the GMM estimator $(\widehat{\btheta}_{opt}, \widehat{\bzeta}_{opt})$, following \cite{Hansen} and \cite{Newey}. It is a typical condition for the application of the central limit theorem to stochastic processes, i.e. to infinite dimensional random vectors. Additionally, in order to make statements about convergence in probability, \ref{cond-diverging-J-3.1} is required to ensure a valid joint probability distribution as the dimension $M$ increases.
\end{remark}
\begin{remark}
Condition \ref{cond-diverging-J-3.2} ensures the covariance of the outcome $\bY_i$ is appropriately controlled as $M \rightarrow \infty$. Alternative conditions may be considered, such as $\alpha$-mixing (\cite{Bradley}), $\rho$-mixing (\cite{Peligrad}), or $\phi$-mixing (\cite{Peligrad}), but this is beyond the scope of this paper. Condition \ref{cond-diverging-J-3.2} can be simplified for the case where $n_k=n$ for all $k=1, \ldots, K$. Then \ref{cond-diverging-J-3.2} becomes $( \sum_{l=1}^K \sum_{j=1}^J \bC_{l,j} )^{-1} \bC_{k,i}=O_p(N^{-\delta_1})$.
\end{remark}
In Theorem \ref{thm:dist-inf2} we show the consistency and asymptotic normality of the DDIMM estimator as $K$ and $J$ diverge to $\infty$.
\begin{theorem}
\label{thm:dist-inf2}
Suppose $N^{-\delta_2}n_{\min}$ and $N^{\delta_3-1/2} KJ$ are bounded as $n_{\min} \rightarrow \infty$ for constants $0 \leq \delta_2 \leq 1$ and $0< \delta_3<1/2$ such that $\delta_3 + \delta_1 + \delta_2/2>1$. Suppose assumptions \ref{consistent}, \ref{psiN-conds}, and \ref{cond-equiv}-\ref{cond-diverging-J-3.3} hold. Let $\bH \in \mathbb{R}^{h \times (p+d)}$ a matrix of rank $r \in \mathbb{N}$, $h \in \mathbb{N}$, $r \leq h$, with finite maximum singular value $\bar{\sigma}(\bH) < \infty$. Let $\bj_{\bH} (\btheta, \bzeta)$ as given in Theorem \ref{thm:dist-inf}. Then as $n_{\min} \rightarrow \infty$,
\begin{align*}
N^{1/2} \bH \left( \begin{array}{c}
\widehat{\btheta}_{DDIMM}-\btheta_0\\
\widehat{\bzeta}_{DDIMM}-\bzeta_0
\end{array} \right) \stackrel{d}{\rightarrow} \mathcal{N} \left(0, \bj_{\bH}(\btheta_0, \bzeta_0) \right).
\end{align*}
\end{theorem}
\begin{proofnoqed}
Write
\begin{align*}
\bH \left( \begin{array}{c}
\widehat{\btheta}_{DDIMM} - \btheta_0\\
\widehat{\bzeta}_{DDIMM} - \bzeta_0 
\end{array} \right)&=
\bH \left( \begin{array}{c}
\widehat{\btheta}_{DDIMM} - \widehat{\btheta}_{opt}\\
\widehat{\bzeta}_{DDIMM} - \widehat{\bzeta}_{opt}
\end{array} \right) + \bH
\left( \begin{array}{c}
\widehat{\btheta}_{opt}- \btheta_0\\
\widehat{\bzeta}_{opt} - \bzeta_0
\end{array} \right).
\end{align*}
To show the asymptotic distribution of the left-hand side, it is sufficient to show that $\bH (
\widehat{\btheta}^T_{DDIMM}-\widehat{\btheta}^T_{opt}, \widehat{\bzeta}^T_{DDIMM}-\widehat{\bzeta}^T_{opt})^T =o_p(N^{-1/2})$.\\
Given the assumptions of the theorem, we have the asymptotic distribution of $(\widehat{\btheta}_{opt}, \widehat{\bzeta}_{opt,ik})$ and $(\widehat{\btheta}_{ik}, \widehat{\bzeta}_{ik})$: both are consistent estimators of $\btheta_0, \bzeta_{ik0}$ and asymptotically normally distributed with rates $N^{-1/2}$ and $n^{-1/2}_k$ respectively. Then for each $k \in \left\{1, \ldots, K \right\}$,
\begin{align*}
\left( \begin{array}{c} 
\widehat{\btheta}_{opt} - \widehat{\btheta}_{ik} \\ 
\widehat{\bzeta}_{opt,ik} - \widehat{\bzeta}_{ik} 
\end{array} \right) &=
\left( \begin{array}{c} 
\widehat{\btheta}_{opt} - \btheta_0 \\ 
\widehat{\bzeta}_{opt,ik} - \bzeta_{ik0}
\end{array} \right) -
\left( \begin{array}{c} 
\widehat{\btheta}_{ik} - \btheta_0 \\ 
 \widehat{\bzeta}_{ik} -\bzeta_{ik0}
\end{array} \right)
= O_p ( n^{-1/2}_k ).
\end{align*}
Defining $\widehat{\bC}^*_{k,i}$ a subset of $\widehat{\bC}_{k,i}$ in Appendix \ref{sec:appendix:technical:Cki}, we can rewrite $(
\widehat{\btheta}^T_{DDIMM}-\widehat{\btheta}^T_{opt}, \widehat{\bzeta}^T_{DDIMM}-\widehat{\bzeta}^T_{opt})^T$ as follows:
\begin{align*}
&\left( \sum \limits_{k=1}^K \sum \limits_{i=1}^J n^2_k \widehat{\bC}_{k,i} \right)^{-1} \left\{ \sum \limits_{k=1}^K \sum \limits_{i=1}^J \left[ n^2_k \widehat{\bC}_{k,i} 
\left( \begin{array}{c} \widehat{\btheta}_{ik}-\widehat{\btheta}_{opt} \\ \widehat{\bzeta}_{list}-\widehat{\bzeta}_{opt} \end{array} \right)
\right] \right\}\\
=& \sum \limits_{k=1}^K \sum \limits_{i=1}^J \left[ \left( \sum \limits_{l=1}^K \sum \limits_{j=1}^J n^2_l \widehat{\bC}_{l,j} \right)^{-1} n^2_k \widehat{\bC}^*_{k,i} 
\left( \begin{array}{c} \widehat{\btheta}_{ik}-\widehat{\btheta}_{opt} \\ \widehat{\bzeta}_{ik}-\widehat{\bzeta}_{opt,ik} \end{array} \right) 
\right] \\
=& \sum \limits_{k=1}^K \sum \limits_{i=1}^J \left[ O_p(N^{-\delta_1}) O_p(n^{-1/2}_k) \right]=O_p(KJN^{-\delta_1}n^{-1/2}_{\min}) \\
=& O_p(N^{1/2-\delta_3} N^{-\delta_1}N^{-\delta_2/2})=O_p(N^{1/2-\delta_3-\delta_1-\delta_2/2})=o_p(N^{-1/2}).
\qed
\end{align*}
\end{proofnoqed}

\vspace{-3em}

\section{SIMULATIONS}
\label{sec:simulations}

In this section we consider two sets of simulations to examine the performance of the closed-form estimator $\widehat{\btheta}_{DDIMM}$ under the linear regression setting $\bmu_i=\bX_i \btheta$, where $\bmu_i=E(\bY_i \lvert \bX_i, \btheta)$ and $\bY_i \sim \mathcal{N}(\bX_i \btheta, \bSigma)$. The first set illustrates the finite sample performance and properties in Theorem \ref{thm:dist-equiv} of $\widehat{\btheta}_{DDIMM}$ with fixed sample size $N$, varying number of subject groups $K$, varying dimensions $M$ of $\bY$, and fixed number of response blocks $J$. The second set of simulations illustrates the performance and properties in Theorem \ref{thm:dist-inf2} of $\widehat{\btheta}_{DDIMM}$ with growing sample size $N$ and response dimension $M$ of $\bY$, and varying number of subjects groups $K$ and response blocks $J$. In both settings, covariates consist of an intercept and two independently simulated $M$-dimensional multivariate normal variables, and the true value of $\btheta$ is set to $\btheta_0=(0.3, 0.6, 0.8)^T$. Simulations are conducted using R software on a standard Linux cluster.\\
We describe the first set of simulations. We specify $\bSigma=\bS \otimes \bA$ with nested correlation structure, where $\otimes$ denotes the Kronecker product, $\bA$ is an AR(1) covariance matrix with standard deviation $\sigma=4$ and correlation $\rho=0.8$, and $\bS$ is a randomly simulated $J \times J$ positive-definite matrix. We consider varying dimensions $M$ of $\bY$ with fixed $J=5$, and a fixed sample size $N=5,000$ with varying $K=1,2,5$. We consider two supervised learning procedures: the pairwise composite likelihood using our own package, and the GEE using R package \verb|geepack| and our own package (see Supplemental Material). With each procedure, 
\begin{figure}[H]
\centerline{\includegraphics[width=\linewidth]{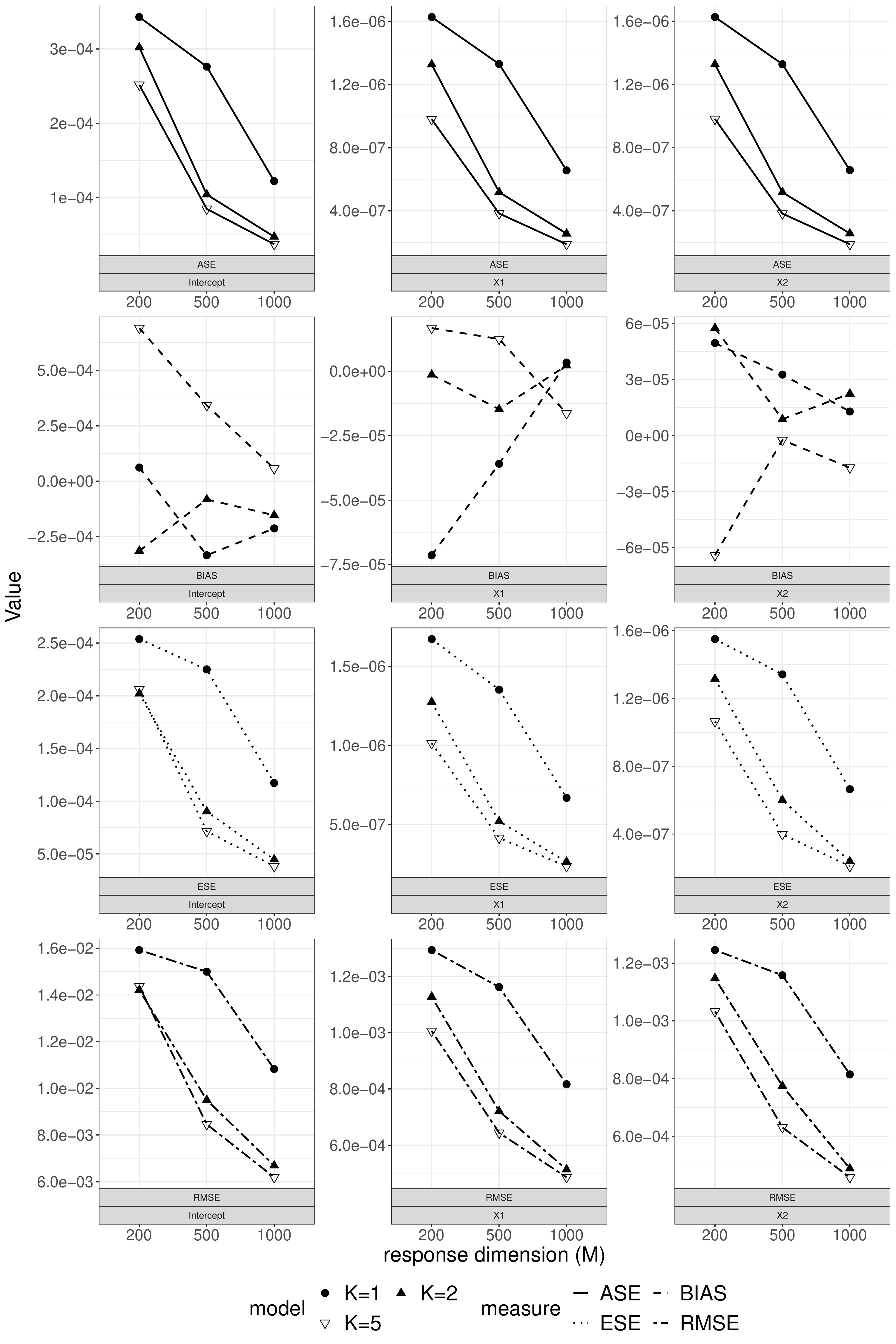}}
\caption{Plot of simulation metrics for GEE, averaged over 1,000 simulations.}\label{GEE-sim-measures}
\end{figure}
\noindent we fit the model with an AR(1) working block correlation structure. Results for the GEE are in Figure \ref{GEE-sim-measures}; results for the pairwise composite likelihood (CL) are in the Supplemental Material. We see that the mean asymptotic standard error (ASE) of $\widehat{\btheta}_{DDIMM}$ approximates the empirical standard error (ESE) for all models, with slight variations due to the type of covariates simulated. This means the covariance formula in Theorem \ref{thm:dist-equiv} is correct. Additionally, $\widehat{\btheta}_{DDIMM}$ appears consistent since root mean squared error (RMSE), ASE and ESE are approximately equal. Moreover, we notice the ASE of $\widehat{\btheta}_{DDIMM}$ decreases as the response dimension $M$ increases. This makes intuitive sense, since an increase in $M$ corresponds to an increase in overall number of observations, resulting in increased power. We also see a decrease in the ASE as the number of groups increases. This is due to the heterogeneity of block covariance parameters. Lastly, we observe from Table \ref{simulations-comp-time} that the mean CPU time is very fast for the GEE, and decreases substantially as the number of subject groups increases.
\begin{table}[H]
\centering
\captionsetup{justification=centering}
\begin{tabular}{r|lll}
  \hline
\multirow{ 2}{*}{Response dimension}  & \multicolumn{3}{l}{Number of subject groups} \\
 & K=1 & K=2 & K=5 \\ 
  \hline
M=200 & 45 & 23 & 11 \\ 
  M=500 & 351 & 184 & 87 \\ 
  M=1,000 & 1956 & 961 & 417 \\ 
   \hline
\end{tabular}
\smallskip
\caption{Mean CPU time in seconds for each setting with the GEE block analysis, averaged over 1,000 simulations. Mean CPU time is computed as the maximum CPU time taken over parallelized block analyses added to the CPU time taken by the rest of the procedure.}
\label{simulations-comp-time}
\end{table}
\noindent We describe the second set of simulations, where we consider diverging sample size $N$ and response dimension $M$, and diverging number of subject groups $K$ and response blocks $J$. We consider two settings: in Setting I, we let the sample size $N=5,000$ with number of response groups $K=1$, and let response dimension $M=4,500$ with number of response blocks $J=6$; in Setting II, we let the sample size $N=10,000$ with number of response groups $K=2$, and let response dimension $M=9,000$ with number of response blocks $J=12$. Responses are simulated from a Multivariate Normal distribution with AR(1) covariance structure, with standard deviation $\sigma=6$ and correlation $\rho=0.8$. This means there are no heterogeneous block parameters, so we expect a slightly less efficient estimator since there is less variability in the outcome. We learn mean and covariance parameters using GEE with an AR(1) working block correlation structure. Mean bias (BIAS), RMSE, ESE and ASE of $\widehat{\btheta}_{DDIMM}$ are in Table \ref{simulations-huge-RMSE-ESE-ASE}. We observe that RMSE, ESE and ASE are very close, indicating appropriate estimation of $\widehat{\btheta}_{DDIMM}$ and its covariance in Theorem \ref{thm:dist-inf2}. We also confirm DDIMM's ability to handle large sample size $N$ and response dimension $M$.
\begin{table}[h!]
\centering
\captionsetup{justification=centering}
\begin{tabular}{rrrrr}
  \hline
Setting & Measure & Intercept & $X_1$ & $X_2$ \\ 
  \hline
I  & RMSE/BIAS & $3.89$/$-1.77$ & $0.64$/$0.09$ & $0.60$/$-0.40$\\
   & ESE/ASE & $3.89$/$3.78$ & $0.64$/$0.59$ & $0.60$/$0.59$\\
II & RMSE/BIAS & $1.86$/$-0.99$ & $0.28$/$-0.03$ & $0.28$/$-0.17$\\
   & ESE/ASE & $1.86$/$1.70$ & $0.28$/$0.27$ & $0.28$/$0.27$\\
  \hline
\end{tabular}
\smallskip
\caption{RMSE$\times 10^{-3}$, BIAS$\times 10^{-4}$, ESE$\times 10^{-3}$, ASE$\times 10^{-3}$ for each setting and each covariate, averaged over 500 simulations.}
\label{simulations-huge-RMSE-ESE-ASE}
\end{table}

\section{DISCUSSION}
\label{sec:discussion}
We have presented the large sample theory as a theoretical guarantee for a Doubly Distributed and Integrated Method of Moments (DDIMM) that incorporates a broad class of supervised learning procedures into a doubly distributed and parallelizable computational scheme for the efficient analysis of large samples of high-dimensional correlated responses in the MapReduce framework. Theoretical challenges related to combining correlated estimators were addressed in the proofs, including the asymptotic properties of the proposed closed-form estimator with fixed and diverging numbers of subject groups and response blocks.\\
The GMM approach to deriving the combined estimator $(\widehat{\btheta}_c, \widehat{\bzeta}_c)$ proposed in \eqref{def:combined-estimator-nuisance} requires only weak regularity of the estimating equations $\bPsi_{jk}$ and $\bG_{jk}$. These assumptions are satisfied by a broad range of learning procedures. The closed-form estimator proposed in equation \eqref{one-step-estimator}, on the other hand, requires local $n^{1/2}_k$-consistent estimators in individual blocks of size $n_k$, which is easily satisfied if $\bPsi_{jk}$ and $\bG_{jk}$ are regular (see \cite{Song} Chapter 3.5 for a definition of regular inference functions). This restricts the class of possible learning procedures, but still includes many analyses of interest.\\
A detailed discussion of the limitations and trade-offs of the single split DIMM with CL block analyses is featured in \cite{Hector-Song}. As mentioned in Section \ref{sec:implementation}, the DDIMM introduces additional flexibility in trading off between computational speed and inference: the number of subject groups $K$ and the smallest block size $n_{\min}$ can be chosen by the investigator to attain the desired speed and efficiency.\\
Particular applications of DDIMM to time series data are immediately obvious. Similarly, we envision potential application to nation-wide hospital daily visit numbers of, for example, asthma patients, over the course of the last decade. One could split the response (hospital daily intake/daily stock price) into $J$ years and into $K$ groups (of hospitals/stocks), analyze blocks separately and in parallel using GEE, and combine results using DDIMM. Finally, extensions of our work to stochastic process modelling are accessible, with more challenging work involving regularization of $\btheta$ also of interest.

\appendix

\section{Technical details}
\label{sec:appendix:technical}

\setcounter{table}{0}
\renewcommand{\thetable}{A.\arabic{table}}

\subsection{Summary of sensitivity matrix formulas}
\label{sec:appendix:technical:sensitivity}

Sensitivity matrices are summarized in Table \ref{sensitivity-summary}.

\begin{table}[h!]
\vspace{0.3cm}
\centering
\captionsetup{justification=centering}
\scalebox{1}{
\begin{tabular}{|c c| |c c c|} 
 \hline
sensitivity of & w.r.t.*  & population & sample & plug-in sample \\ [0.5ex] 
 \hline\hline
$\bpsi_{i,jk}$ & $\btheta$ & $\bs^{\btheta}_{\bpsi_{jk}} (\btheta, \bzeta_{jk})$ & $\bS^{\btheta}_{\bpsi_{jk}} (\btheta, \bzeta_{jk})$ & $\widehat{\bS}^{\btheta}_{\bpsi_{jk}}=\bS^{\btheta}_{\bpsi_{jk}} (\widehat{\btheta}_{jk}, \widehat{\bzeta}_{jk})$ \\
$\bpsi_{i,jk}$ & $\bzeta_{jk}$ & $\bs^{\bzeta}_{\bpsi_{jk}}(\btheta, \bzeta_{jk})$ &$\bS^{\bzeta}_{\bpsi_{jk}} (\btheta, \bzeta_{jk})$ & $\widehat{\bS}^{\bzeta}_{\bpsi_{jk}}=\bS^{\bzeta}_{\bpsi_{jk}} (\widehat{\btheta}_{jk}, \widehat{\bzeta}_{jk})$ \\ 
$\bg_{i,jk}$ & $\btheta$ & $\bs^{\btheta}_{\bg_{jk}} (\btheta, \bzeta_{jk})$ & $\bS^{\btheta}_{\bg_{jk}} (\btheta, \bzeta_{jk})$ & $\widehat{\bS}^{\btheta}_{\bg_{jk}}=\bS^{\btheta}_{\bg_{jk}} (\widehat{\btheta}_{jk}, \widehat{\bzeta}_{jk})$ \\
$\bg_{i,jk}$ & $\bzeta_{jk}$ & $\bs^{\bzeta}_{\bg_{jk}} (\btheta, \bzeta_{jk})$ & $\bS^{\bzeta}_{\bg_{jk}} (\btheta, \bzeta_{jk})$ & $\widehat{\bS}^{\bzeta}_{\bg_{jk}}=\bS^{\bzeta}_{\bg_{jk}} (\widehat{\btheta}_{jk}, \widehat{\bzeta}_{jk})$ \\
$\veco \left( \bpsi_{i,jk}, \bg_{i,jk} \right)$ & $(\btheta, \bzeta_{jk})$ & $\bs_{jk} (\btheta, \bzeta_{jk})$ & $\bS_{jk} (\btheta, \bzeta_{jk})$ & $\widehat{\bS}_{jk}=\bS_{jk} (\widehat{\btheta}_{jk}, \widehat{\bzeta}_{jk})$ \\[1ex] 
 \hline
\end{tabular}
}
\smallskip
\caption{Summary of sensitivity formulas. Formulas that are not used are marked ``---''.\\
*``w.r.t.'' shorthand for ``with respect to''. }
\label{sensitivity-summary}
\end{table}


\subsection{Subsetting operation on variability matrices}
\label{sec:appendix:technical:subset}

Operation $\left[\widehat{\bV}^{\bpsi}_N \right]_{ij:k}$ extracts a submatrix of $\widehat{\bV}^{\bpsi}_N$ consisting of rows $\left\{ (i-1)+(k-1)J \right\} p+1$ to $\left\{ i+(k-1)J\right\} p$ and columns $\left\{ j-1+(k-1)J \right\} p+1$ to $\left\{ j+(k-1)J \right\} p$. Operation $\left[ \widehat{\bV}^{\bg}_N \right]_{ij:k}$ extracts a submatrix of $\widehat{\bV}^{\bg}_N$ consisting of rows $1+D^{ik}$ to $d_{ik}+D^{ik}$ and columns $1+D^{jk}$ to $d_{jk} + D^{jk}$. Operation $\left[ \widehat{\bV}^{\bpsi \bg}_N \right]_{ij:k}$ extracts a submatrix of $\widehat{\bV}^{\bpsi \bg}_N$ consisting of rows $\left\{ (i-1)+(k-1)J \right\} p+1$ to $\left\{ i+(k-1)J \right\} p$ and columns $1+D^{jk}$ to $d_{jk} + D^{jk}$, where $d_{jk}$ is the dimension of $\bzeta_{jk}$ and $D^{jk}$ is defined in Section \ref{subsec:implementation:def-C_ki}.

\subsection{Cumulative sum of dimensions of $\bzeta$}
\label{sec:appendix:technical:Dik}

Recall that we define $D^{ik}$ as the sum of the dimensions of $\bzeta_{11}, \ldots, \bzeta_{i-1k}$, and $D^k$ as the sum of the dimensions of $\bzeta_{11}, \ldots, \bzeta_{Jk-1}$. Specifically, let $D^{ik}= \sum_{l=1}^{k-1} \sum_{j=1}^{J} d_{j l}+\sum_{j=1}^{i-1} d_{jk}$ for $i,k>1$, $D^{1k}= \sum_{l=1}^{k-1} \sum_{j=1}^{J} d_{j l}$ for $k>1$, and $D^{11}=0$. Let $D^k= \sum_{l=1}^{k-1} d_{l}$ for $k>1$ and $D^1=0$.

\subsection{Definition of $\widehat{\bC}^*_{k,i}$}
\label{sec:appendix:technical:Cki}

Let $k \in \left\{1, \ldots, K \right\}$ and $i \in \left\{1, \ldots, J \right\}$. Recall the definitions of $\widehat{\bA}^{\btheta}_{k,ij}$, $\widehat{\bA}^{\bzeta}_{k,ij}$, $\widehat{\bB}^{\btheta}_{k,ij}$ and $\widehat{\bB}^{\bzeta}_{k,ij}$ in Section \ref{subsec:implementation:def-C_ki}. Define
\begin{align*}
\widehat{\bC}^*_{k,i}&=\left( \begin{array}{cc}
\sum \limits_{j=1}^J \widehat{\bA}^{\btheta}_{k,ij} &\sum \limits_{j=1}^J \widehat{\bA}^{\bzeta}_{k,ij}  \\
\multicolumn{2}{@{}c@{\quad}}{\boldsymbol{0}_{D^{ik} \times (p+d)}}\\
\widehat{\bB}^{\btheta}_{k,i1} & \widehat{\bB}^{\bzeta}_{k,i1} \\
\multicolumn{2}{@{}c@{\quad}}{ ~~\vdots }\\
\widehat{\bB}^{\btheta}_{k,iJ} &\widehat{\bB}^{\bzeta}_{k,iJ}\\
\multicolumn{2}{@{}c@{\quad}}{ \boldsymbol{0}_{(d-d_{ik}-D^{ik}) \times (p+d)}}
\end{array} \right).
\end{align*}

\section{Additional proofs}
\label{sec:appendix:proofs}

\subsection{Proof of Theorem \ref{thm:dist-equiv}:}
\label{subsec:appendix:proofs:dist-equiv}
The following lemmas complete the proof of Theorem \ref{thm:dist-equiv} given in the paper, under the assumed conditions.
\begin{thmlemma}
\label{appendix:lemma:lambda}
Define $\lambda(\btheta, \bzeta)$ as in \eqref{def:lambda} in the proof of Theorem \ref{thm:dist-equiv}. Then $\lambda(\btheta_0, \bzeta_0) \stackrel{p}{\rightarrow} 0$ as $n_{\min} \rightarrow \infty$.
\end{thmlemma}
\begin{proofnoqed}
Using Lemma \ref{lemma:add-results},
\begin{align*}
\lambda(\btheta_0, \bzeta_0)
&=\frac{1}{N^2} \sum \limits_{k=1}^K \sum \limits_{i=1}^J n^2_k \widehat{\bC}_{k,i} \left( \begin{array}{c} \btheta_0 - \widehat{\btheta}_{ik} \\ \bzeta_0 - \widehat{\bzeta}_{list} \end{array} \right)\\
&= O_p \left( n^{-1/2}_{\min} \right) \left\{ \bj(\btheta_0, \bzeta_0)+O_p\left( N^{-1/2} \right) \right\}\\
&=O_p\left( n^{-1/2}_{\min} \right) + O_p \left( n^{-1/2}_{\min} N^{-1/2} \right) \stackrel{p}{\rightarrow} 0 \mbox{ as }n_{\min} \rightarrow \infty. \qed
\end{align*}
\end{proofnoqed}

\vspace{-2em}

\begin{thmlemma}
\label{appendix:lemma:Taylor}
The following relationship holds:
\begin{align*}
\left( \begin{array}{c}
\bPsi_{jk}(\btheta_0; \bzeta_{jk0})\\
\bG_{jk}(\bzeta_{jk0}; \btheta_0)
\end{array} \right) &=\widehat{\bS}_{jk} \left( \begin{array}{c}
\widehat{\btheta}_{jk}-\btheta_0\\
\widehat{\bzeta}_{jk}-\bzeta_{jk0}
\end{array} \right) + O_p(n^{-1}_k).
\end{align*}
\end{thmlemma}

\begin{proofnoqed}
Let $j \in \left\{1, \ldots, J \right\}$, $k \in \left\{1, \ldots, K\right\}$ fixed. For convenience, denote
\begin{align*}
\bTau_{jk}(\btheta, \bzeta_{jk})&=\left( \begin{array}{c} 
\bPsi_{jk}(\btheta; \bzeta_{jk}) \\ \bG_{jk}(\bzeta_{jk}; \btheta)
\end{array} \right),~~
\btau_{i,jk}(\btheta, \bzeta_{jk})=\left( \begin{array}{c} 
\bpsi_{i,jk}(\btheta; \bzeta_{jk}) \\ \bg_{i,jk}(\bzeta_{jk}; \btheta)
\end{array} \right).
\end{align*}
By first-order Taylor expansion,
\begin{align}
E_{\btheta, \bzeta_{jk}} &\left\{ \btau_{i,jk}(\widehat{\btheta}_{jk}, \widehat{\bzeta}_{jk}) \right\}= 
E_{\btheta, \bzeta_{jk}} \left\{ \btau_{i,jk}(\btheta_0, \bzeta_{jk0}) \right\}+ \nonumber \\
&~~~~~~~~~~~~~~~~~~
\nabla_{\btheta} E_{\btheta, \bzeta_{jk}} \left\{ \btau_{i,jk}(\btheta, \bzeta_{jk})\right\} \rvert_{\btheta^*, \bzeta^*_{jk}} \left(\begin{array}{c}
\widehat{\btheta}_{jk}-\btheta_0 \\
\widehat{\bzeta}_{jk}-\bzeta_{jk0}
\end{array} \right),
\label{equiv-2}
\end{align}
where $(\btheta^*, \bzeta^*_{jk})$ lies between $(\btheta_0, \bzeta_{jk0})$ and $(\widehat{\btheta}_{jk}, \widehat{\bzeta}_{jk})$. By condition \ref{cond-equiv}, 
\begin{align}
\bTau_{jk}(\widehat{\btheta}_{jk}, \widehat{\bzeta}_{jk}) &- 
\bTau_{jk}(\btheta_0, \bzeta_{jk0}) - 
E_{\btheta, \bzeta_{jk}} \left\{ \btau_{i,jk}(\widehat{\btheta}_{jk}, \widehat{\bzeta}_{jk}) \right\} \nonumber \\
&= O_p(N^{-1/2}) \frac{1+N^{1/2} O_p(n^{-1/2}_k) }{N^{1/2}} =O_p(n^{-1/2}_k N^{-1/2}).
\label{equiv-3}
\end{align}
In other words, the norm of the difference between $\bTau_{jk}(\btheta_0, \bzeta_{jk0}) $ and $\bTau_{jk}(\widehat{\btheta}_{jk}, \allowbreak \widehat{\bzeta}_{jk}) - 
E_{\btheta, \bzeta_{jk}} \{ \btau_{i,jk}(\widehat{\btheta}_{jk}, \widehat{\bzeta}_{jk}) \}$ goes to $0$ at a rate faster than $(N n_k)^{-1/2}$. Adding \eqref{equiv-2} and \eqref{equiv-3}, we have
\begin{align*}
-\bTau_{jk}&(\btheta_0, \bzeta_{jk0})=\bTau_{jk}(\widehat{\btheta}_{jk}, \widehat{\bzeta}_{jk})-\bTau_{jk}(\btheta_0, \bzeta_{jk0}) -E_{\btheta, \bzeta_{jk}} \btau_{i,jk}(\btheta_0, \bzeta_{jk0})\\
&=\nabla_{\btheta} E_{\btheta, \bzeta_{jk}} \btau_{i,jk}(\btheta, \bzeta_{jk}) \rvert_{\btheta^*, \bzeta^*_{jk}}\left( \begin{array}{c}
\widehat{\btheta}_{jk}-\btheta_0\\
\widehat{\bzeta}_{jk}-\bzeta_{jk0}
\end{array} \right) + O_p(n^{-1/2}_k N^{-1/2}) \\
&=-\bs_{jk}(\btheta^*, \bzeta^*_{jk}) \left( \begin{array}{c}
\widehat{\btheta}_{jk}-\btheta_0\\
\widehat{\bzeta}_{jk}-\bzeta_{jk0}
\end{array} \right) + O_p(n^{-1/2}_k N^{-1/2}).
\end{align*}
Rearranging yields
\begin{align}
\bTau_{jk}(\btheta_0, \bzeta_{jk0})=\bs_{jk}(\btheta^*, \bzeta^*_{jk}) \left( \begin{array}{c}
\widehat{\btheta}_{jk}-\btheta_0\\
\widehat{\bzeta}_{jk}-\bzeta_{jk0}
\end{array} \right) + O_p(n^{-1/2}_k N^{-1/2}). \label{equiv-4}
\end{align}
Finally, note that $\widehat{\bS}_{jk}=\bs_{jk}(\btheta_0, \bzeta_{jk0})+O_p(n^{-1/2}_k)=\bs_{jk}(\btheta^*, \bzeta^*_{jk}) + O_p(n^{-1/2}_k)$. Then plugging this into \eqref{equiv-4}, we have:
\begin{align*}
\bTau_{jk}(\btheta_0, \bzeta_{jk0})&=\left( \widehat{\bS}_{jk} + O_p(n^{-1/2}_k) \right)
\left( \begin{array}{c}
\widehat{\btheta}_{jk}-\btheta_0\\
\widehat{\bzeta}_{jk}-\bzeta_{jk0}
\end{array} \right)
+ O_p(n^{-1/2}_k N^{-1/2}) \\
&=\widehat{\bS}_{jk} \left( \begin{array}{c}
\widehat{\btheta}_{jk}-\btheta_0\\
\widehat{\bzeta}_{jk}-\bzeta_{jk0}
\end{array} \right) + O_p(n^{-1}_k). \qed
\end{align*}
\end{proofnoqed}

\vspace{-2em}

\subsection{Proof of Theorem \ref{thm:dist-inf}}
\label{subsec:appendix:proofs:dist-inf}

The following lemmas complete the proof of Theorem \ref{thm:dist-inf} given in the paper, under the assumed conditions.
\begin{thmlemma}
\label{thm:dist-inf:lemma-1}
Define $\lambda(\btheta, \bzeta)$ as in \eqref{def:lambda} in the proof of Theorem \ref{thm:dist-equiv}. Then $\left\| \lambda(\btheta_0, \bzeta_0) \right\| = O_p(N^{-1/2-\delta} n^{1/2}_{\max})$ and $\left\| \left\{ \nabla_{\btheta, \bzeta} \lambda(\btheta, \bzeta) \right\}^{-1} \right\| = O_p\left( N^{1/2+\delta} n^{-1}_{\max} \right)$.
\end{thmlemma}
\begin{proof}
Due to the independence between subject groups, $\widehat{\bV}^{\bpsi}_N$, $\widehat{\bV}^{\bpsi \bg}_N$ and $\widehat{\bV}^{\bg}_N$ are all block diagonal: $\widehat{\bV}^{\bpsi}_N=\mbox{diag} \left\{ \widehat{\bV}^{\bpsi}_k \right\}_{k=1}^K$, $\widehat{\bV}^{\bpsi \bg}_N=\mbox{diag} \left\{ \widehat{\bV}^{\bpsi \bg}_k \right\}_{k=1}^K$, and $\widehat{\bV}^{\bg}_N=\mbox{diag} \left\{ \widehat{\bV}^{\bg}_k \right\}_{k=1}^K$. By the independence of subject groups, let
\begin{align*}
\bv^{-1}(\btheta, \bzeta)&=\left( \begin{array}{cc}
\bv^{\bpsi}(\btheta, \bzeta) & \bv^{\bpsi \bg}(\btheta, \bzeta)\\
\bv^{\bpsi \bg ~T}(\btheta, \bzeta) & \bv^{\bg}(\btheta, \bzeta)
\end{array} \right)\\
&=\left( \begin{array}{cc}
\mbox{diag} \left\{ \frac{N}{n_k} \bv^{\bpsi}_k(\btheta, \bzeta)\right\}_{k=1}^K & \mbox{diag} \left\{ \frac{N}{n_k} \bv^{\bpsi \bg}_k(\btheta, \bzeta)\right\}_{k=1}^K\\
\mbox{diag} \left\{ \frac{N}{n_k} \bv^{\bpsi \bg ~T}_k(\btheta, \bzeta)\right\}_{k=1}^K & \mbox{diag} \left\{ \frac{N}{n_k} \bv^{\bg}_k(\btheta, \bzeta) \right\}_{k=1}^K
\end{array} \right).
\end{align*}
Similar to the proof of Lemma \ref{lemma:add-results}, it can easily be shown that for each $k=1, \ldots, K$, $\widehat{\bV}^{\bpsi}_k= (N/n_k) \bv^{\bpsi}_k(\btheta_0, \bzeta_0) + O_p(N^{-1/2})$, $\widehat{\bV}^{\bpsi \bg}_k=(N/n_k) \bv^{\bpsi \bg}_k(\btheta_0, \bzeta_0) + O_p(N^{-1/2})$, and $\widehat{\bV}^{\bg}_k=(N/n_k) \bv^{\bg}_k(\btheta_0, \bzeta_0) + O_p(N^{-1/2})$. Consider an arbitrary $k \in \left\{1, \ldots, K \right\}$. Let $(N/n_k)\left[ \bv^{\bpsi}_k (\btheta_0, \bzeta_0) \right]_{ji}=\left[ \bv^{\bpsi} (\btheta_0, \bzeta_0) \right]_{ji:k}$, and similarly define $\left[ \bv^{\bpsi \bg}_k (\btheta_0, \bzeta_0) \right]_{ji}$ and $\left[ \bv^{\bg}_k (\btheta_0, \bzeta_0) \right]_{ji}$. Then $\widehat{\bA}^{\btheta}_{k,ij}=(N/n_k) \{ \ba^{\btheta}_{k,ij}+\allowbreak O_p(n^{-1/2}_k) \}$, where $\ba^{\btheta}_{k,ij}$ is defined as\\
\scalebox{0.95}{\parbox{\linewidth}{%
\begin{align*}
&\left\{ 
\bs^{\btheta~T}_{\bpsi_{jk}}(\btheta_0, \bzeta_{jk0})\left[\bv^{\bpsi}_k(\btheta_0, \bzeta_0)\right]_{ji} + 
\bs^{\btheta~T}_{\bg_{jk}}(\btheta_0, \bzeta_{jk0})\left[\bv^{\bpsi \bg~ T}_k(\btheta_0, \bzeta_0)\right]_{ji} 
\right\} 
\bs^{\btheta}_{\bpsi_{ik}}(\btheta_0, \bzeta_0) + \\
&\left\{ 
\bs^{\btheta~T}_{\bpsi_{jk}}(\btheta_0, \bzeta_{jk0})\left[\bv^{\bpsi \bg}_k(\btheta_0, \bzeta_0)\right]_{ji} + 
\bs^{\btheta~T}_{\bg_{jk}}(\btheta_0, \bzeta_{jk0})\left[\bv^{\bg}_k(\btheta_0, \bzeta_0)\right]_{ji} 
\right\} 
\bs^{\btheta}_{\bg_{ik}}(\btheta_0, \bzeta_0).
\end{align*}
}}\\
We can show similar results for $\widehat{\bA}^{\bzeta}_{k,ij} $, $\widehat{\bB}^{\btheta}_{k,ij}$ and $\widehat{\bB}^{\bzeta}_{k,ij}$.
Then we can rewrite\\
\scalebox{0.95}{\parbox{\linewidth}{%
\begin{align*}
\left\| \lambda(\btheta_0, \bzeta_0) \right\| &\leq 
\sum \limits_{k=1}^K O_p(n^{1/2}_k N^{-1})=O_p(K n^{1/2}_{\max} N^{-1})=O_p(N^{-1/2-\delta} n^{1/2}_{\max}), \mbox{ and}\\
\left\| \nabla_{\btheta, \bzeta} \lambda(\btheta, \bzeta) \right\| &\leq \frac{1}{N^2} \sum \limits_{k=1}^K \sum \limits_{i=1}^J n^2_k \left\| \widehat{\bC}_{k,i} \right\| \\
&\leq O_p\left( N^{-1/2-\delta} n^{1/2}_{\max} \right) + O\left( N^{-1/2-\delta} n_{\max} \right)=O_p\left( N^{-1/2-\delta} n_{\max} \right).
\end{align*}
}}\\
Since $\nabla_{\btheta, \bzeta} \lambda(\btheta, \bzeta)$ is symmetric positive-definite, the above provides a bound on its eigenvalues. Therefore, $\left\| \left\{ \nabla_{\btheta, \bzeta} \lambda(\btheta, \bzeta) \right\}^{-1} \right\| = O_p\left( N^{1/2+\delta} n^{-1}_{\max} \right)$.
\end{proof}

\begin{thmlemma}
\label{thm:dist-inf:lemma-2}
For some matrices $\bE_k$, $k=1, \ldots, K$, of $\boldsymbol{0}$'s and $\boldsymbol{1}$'s, the following asymptotic properties hold:
\begin{align*}
\frac{n^2_k}{N^2} \sum \limits_{i=1}^J \widehat{\bC}_{k,i}\left( \begin{array}{c} \widehat{\btheta}_{ik}-\btheta_0 \\ \widehat{\bzeta}_{list} - \bzeta_0 \end{array} \right)
&=\frac{n_k}{N} \bE_k \bZ_k + O_p\left( N^{-1} \right),\\
\mbox{and}~~~\frac{n^2_k}{N^2} \sum \limits_{i=1}^J \widehat{\bC}_{k,i}&=\frac{n_k}{N} \bE_k \bj_k(\btheta_0, \bzeta_{k0}) \bE^T_k + O_p\left( n^{1/2}_kN^{-1} \right),
\end{align*}
where $n^{1/2}_k  \bZ_k \stackrel{d}{\rightarrow} \mathcal{N} \left(\boldsymbol{0}, \bj^{-1}_k(\btheta_0, \bzeta_{k0}) \right)$.
\end{thmlemma}

\begin{proofnoqed}
Recall that $\widehat{\bC}_{k,i}(\widehat{\btheta}^T_{ik}-\btheta^T_0, \widehat{\bzeta}^T_{list}-\bzeta^T_0)^T=\widehat{\bC}^*_{k,i}(\widehat{\btheta}^T_{ik}-\btheta^T_0, \widehat{\bzeta}^T_{ik}- \bzeta^T_{ik0})^T$. Let $\left[ \bv^{-1}_k(\btheta, \bzeta_k) \right]_{ij}$ subset the rows for the parameters corresponding to block $(i,k)$ and the columns for the parameters corresponding to block $(j,k)$ of matrix $\bv^{-1}_k(\btheta, \bzeta_k)$. Define $\bj_{jik}(\btheta, \bzeta_{jk}, \bzeta_{ik})=\bs_{jk}(\btheta, \bzeta_{jk}) \left[ \bv^{-1}_k(\btheta, \bzeta_{k}) \right]_{ji} \allowbreak \bs_{ik}(\btheta, \bzeta_{ik})$, and $\left[ \bj^{-1}_k(\btheta_0, \bzeta_{k0})\right]_i$ the submatrix of $\bj^{-1}_k(\btheta_0, \bzeta_{k0})$ corresponding to parameters in block $(i,k)$, such that
\begin{align*}
n^{1/2}_k \left\{ \sum \limits_{j=1}^J \bj_{jik}(\btheta_0, \bzeta_{jk0}, \bzeta_{ik0}) \right\} \left( \begin{array}{c}
\widehat{\btheta}_{ik} -\btheta_0 \\
\widehat{\bzeta}_{ik} - \bzeta_{ik0}
\end{array} \right) \stackrel{d}{\rightarrow} \mathcal{N} \left(\boldsymbol{0}, \left[ \bj^{-1}_k(\btheta_0, \bzeta_{k0})\right]_i \right).
\end{align*}
Then using the results in the proof of Lemma \ref{thm:dist-inf:lemma-1}, let $\bE_k$ and $\bE_{k,i}$ matrices of $\boldsymbol{0}$'s and $\boldsymbol{1}$'s such that 
\begin{align*}
&\frac{n^2_k}{N^2} \sum \limits_{i=1}^J \widehat{\bC}_{k,i}=\frac{n_k}{N} \bE_k \left\{ \bj_k(\btheta_0, \bzeta_{k0}) + O_p\left( n^{-1/2}_k \right) \right\} \bE^T_k\\
&~~~=\frac{n_k}{N} \bE_k \bj_k(\btheta_0, \bzeta_{k0}) \bE^T_k + O_p\left( n^{1/2}_kN^{-1} \right), \mbox{ and}\\
&\frac{n^2_k}{N^2} \sum \limits_{i=1}^J \widehat{\bC}_{k,i}\left( \begin{array}{c} \widehat{\btheta}_{ik}-\btheta_0 \\ \widehat{\bzeta}_{list} - \bzeta_0 \end{array} \right)\\
&~~~=\frac{n_k}{N} \bE_k \sum \limits_{i=1}^J \bE_{k,i} \left\{ \sum \limits_{j=1}^J \bj_{jik}(\btheta_0, \bzeta_{jk0}, \bzeta_{ik0}) + O_p\left( n^{-1/2}_k \right) \right\}
\left( \begin{array}{c}
\widehat{\btheta}_{ik} -\btheta_0 \\
\widehat{\bzeta}_{ik} - \bzeta_{ik0}
\end{array} \right)\\
&~~~=\frac{n_k}{N} \bE_k \sum \limits_{i=1}^J \bE_{k,i} \sum \limits_{j=1}^J \bj_{jik}(\btheta_0, \bzeta_{jk0}, \bzeta_{ik0}) 
\left( \begin{array}{c}
\widehat{\btheta}_{ik} -\btheta_0 \\
\widehat{\bzeta}_{ik} - \bzeta_{ik0}
\end{array} \right) + O_p\left( N^{-1} \right).
\end{align*}
To obtain the desired result, define 
\begin{align*}
\bZ_k&=\sum \limits_{i=1}^J \bE_{k,i} \sum \limits_{j=1}^J \bj_{jik}(\btheta_0, \bzeta_{jk0}, \bzeta_{ik0}) 
\left( \begin{array}{c}
\widehat{\btheta}_{ik} -\btheta_0 \\
\widehat{\bzeta}_{ik} - \bzeta_{ik0}
\end{array} \right).  \qed
\end{align*}
\end{proofnoqed}

\vspace{-1em}

\begin{thmlemma}
\label{thm:dist-inf:lemma-3}
$N^{1/2} \bH \left( \widehat{\btheta}^T_{DDIMM}-\btheta^T_0,
\widehat{\bzeta}^T_{DDIMM}- \bzeta^T_0  \right)$ can be rewritten as\\
\scalebox{0.95}{\parbox{\linewidth}{%
\begin{align*}
\bH \left\{ \sum \limits_{k=1}^K \frac{n_k}{N} \bE_k \bj_k(\btheta_0, \bzeta_{k0}) \bE^T_k + O_p\left(n^{1/2}_{\max} N^{-1/2-\delta}\right) \right\}^{-1} \left[ \sum \limits_{k=1}^K \left\{ \left( \frac{n_k}{N} \right)^{1/2} \bE_k n^{1/2}_k \bZ_k \right\} + O_p \left(N^{-\delta} \right) \right]&.
\end{align*}
}}
\end{thmlemma}

\begin{proofnoqed}
\begin{align*}
&N^{1/2} \bH \left( \begin{array}{c}
\widehat{\btheta}_{DDIMM}-\btheta_0\\
\widehat{\bzeta}_{DDIMM}- \bzeta_0 
\end{array} \right) \\
&=N^{1/2} \bH \left( \sum \limits_{k=1}^K \sum \limits_{i=1}^J \frac{n^2_k}{N^2} \widehat{\bC}_{k,i} \right)^{-1} \sum \limits_{k=1}^K \sum \limits_{i=1}^J \frac{n^2_k}{N^2} \widehat{\bC}_{k,i} \left( \begin{array}{c} \widehat{\btheta}_{ik}-\btheta_0 \\ \widehat{\bzeta}_{list} - \bzeta_0 \end{array} \right) \\
&=\bH \left[ \sum \limits_{k=1}^K \left\{\frac{n_k}{N} \bE_k \bj_k(\btheta_0, \bzeta_{k0}) \bE^T_k + O_p(n^{1/2}_kN^{-1}) \right\} \right]^{-1} \cdot \\
&~~~~~\sum \limits_{k=1}^K \left\{\frac{n_k}{N^{1/2}} \bE_k \sum \limits_{i=1}^J \bE_{k,i} \bj_{ik}(\btheta_0, \bzeta_{jk0}, \bzeta_{ik0}) 
\left( \begin{array}{c}
\widehat{\btheta}_{ik} -\btheta_0 \\
\widehat{\bzeta}_{ik} - \bzeta_{ik0}
\end{array} \right) + O_p(N^{-1/2}) \right\} \\
&=\bH \left\{ \sum \limits_{k=1}^K \frac{n_k}{N} \bE_k \bj_k(\btheta_0, \bzeta_{k0}) \bE^T_k+ O_p\left(K n^{1/2}_{\max}N^{-1}\right) \right\}^{-1} \cdot \\
&~~~~~\left[ \sum \limits_{k=1}^K \left\{\frac{n_k}{N^{1/2}} \bE_k \sum \limits_{i=1}^J \bE_{k,i} \bj_{ik}(\btheta_0, \bzeta_{jk0}, \bzeta_{ik0}) 
\left( \begin{array}{c}
\widehat{\btheta}_{ik} -\btheta_0 \\
\widehat{\bzeta}_{ik} - \bzeta_{ik0}
\end{array} \right) \right\} + O_p \left(KN^{-1/2} \right) \right]\\
&=\bH \left\{ \sum \limits_{k=1}^K \frac{n_k}{N} \bE_k \bj_k(\btheta_0, \bzeta_{k0}) \bE^T_k + O_p\left(n^{1/2}_{\max} N^{-1/2-\delta}\right) \right\}^{-1} \cdot \\
&~~~~~\left[ \sum \limits_{k=1}^K \left\{ \left( \frac{n_k}{N} \right)^{1/2} \bE_k n^{1/2}_k \bZ_k \right\} + O_p \left(N^{-\delta} \right) \right]. \qed
\end{align*}
\end{proofnoqed}

\vspace{-2em}

\bibliographystyle{apalike}

\bibliography{DDIMM-bib-20191202}

\end{document}